\documentclass[11pt]{article}
\usepackage{amsmath,amsthm,amssymb,amsfonts}
\usepackage{float}
\usepackage{color,fullpage}
\usepackage{cite}
\usepackage[title,titletoc]{appendix}
\usepackage{algorithm,algorithmic}
\usepackage{tikz}
\usepackage[overload]{empheq}
\usepackage[numbers, compress]{natbib}

\usepackage{amsthm}
\newcommand{\RR}{\mathbb{R}}

\newcommand{\dom}{{\mathrm{dom}}\,} % domain
\newcommand{\inte}{{\mathrm{int}}}

\newcommand{\cB}{{\mathcal{B}}}
\newcommand{\cX}{{\mathcal{X}}}
\newcommand{\cY}{{\mathcal{Y}}}

\newcommand{\cU}{{\mathcal{U}}}

\newtheorem{assumption}{Assumption}

\newtheorem{theorem}{Theorem}
\newtheorem{lemma}{Lemma}
\newtheorem{definition}{Definition}

\newtheorem{remark}{Remark}
\newtheorem{proposition}{Proposition}

\makeatletter
\begin{document}

\title{Mirror frameworks for relatively Lipschitz and monotone-like variational inequalities}

\author{Hui Zhang\thanks{
Department of Mathematics, National University of Defense Technology,
Changsha, Hunan 410073, China.  Email: \texttt{h.zhang1984@163.com}
}
\and Yu-Hong Dai\thanks{100190 $\&$ School of Mathematical Sciences,
Chinese Academy of Sciences, Beijing 100049, China, Email: \texttt{dyh@lsec.cc.ac.cn}
}
}

\date{\today}

\maketitle

\begin{abstract}
Nonconvex-nonconcave saddle-point optimization in machine learning has triggered lots of research for studying non-monotone variational inequalities (VI). In this work, we introduce two mirror frameworks, called mirror extragradient method and mirror extrapolation method, for approximating solutions to relatively Lipschitz and monotone-like VIs. The former covers the well-known Nemirovski's mirror prox method and Nesterov's dual extrapolation method, and the recently proposed Bregman extragradient method; all of them can be reformulated into a scheme that is very similar to the original form of extragradient method. The latter includes the operator extrapolation method and the Bregman extrapolation method as its special cases. The proposed mirror frameworks allow us to present a unified and improved convergence analysis for all these existing methods under relative Lipschitzness and monotone-like conditions that may be the currently weakest assumptions guaranteeing (sub)linear convergence.
\end{abstract}



%% \linenumbers

%% main text\
\section{Introduction}

Variational inequalities (VI), as a powerful mathematical tool to reformulate optimality conditions and saddle-point optimization \cite{2003Finite,2018a}, have recently attracted renewed attention in the communities of machine learning and optimization.  The most classic method for solving VIs may be the extragradient (EG) method, which was initially proposed by Korpelevich to modify the projected gradient method using extrapolation techniques so that the required strongly monotone assumption can be relaxed \cite{1976An}. From a mirror descent perspective, Nemirovski \cite{2004Prox} proposed a prox-method, which generalizes the EG method from the standard Euclidean case to non-Euclidean cases. Another non-Euclidean extension of EG is Nesterov's dual extrapolation (EP) proposed via dual averaging technique \cite{2007dual}. Recently, the authors of \cite{2019Unifying} unified the mirror descent and dual averaging and proposed a unified framework which covers Nemirovski's and Nesterov's methods. Besides these extensions and unifications, there are lots of research aimed at improving the original EG methods by reducing the cost of EG per iteration that needs to evaluate gradients twice; see e.g. \cite{1980A,1998A,2011The,2015projected,2020A}.

A common assumption behind all these methods above is that the operator in VIs is Lipschitz continuous and monotone. However, modern nonconvex-nonconcave saddle-point optimization problems such as those appeared in deep learning go beyond monotone VIs and hence the existing results fail to apply in non-monotone settings. This motivates a surge of interest to generalized VIs and its associated algorithms \cite{2015On,2017stochastic,JMLR:v22:20-533,2020efficient,2020simple,lee2021fast,grimmer2021landscape}. We restrict our attention to the line of research that relaxes the Lipschitz continuity and monotonicity assumptions.

The concept of relative Lipschitzness appeared about five years ago in the field of convex optimization \cite{Bauschke2016A,Lu2016relatively}. Following a basic observation on the classic proximal gradient, the authors of \cite{Bauschke2016A} proposed the Lipschitz-like/convexity condition to replace the usual Lipschitz continuity property. This property was also independently discovered by the authors of \cite{Lu2016relatively} under the name of relative smoothness. Recently, the notion of relative smoothness was extended to monotone VIs in \cite{2020Relative}. The extension is not trivial since the formulation in VIs is quite different from that in optimization. It has been illustrated that the relative Lipschitzness is a very suitable tool to convergence analysis for several extragradient-type methods applied to monotone VIs \cite{2020Relative,2021Extragradient}.

Regarding to the relaxation of monotonicity, generalized monotone VIs, whose operators are not necessarily monotone, were considered in the past few years; see e.g. the work \cite{2015On}. In the community of machine learning, the generalized monotone assumption \cite{2015On} was rediscovered in the work \cite{2017stochastic} under the name of variational coherence. In order to study weakly-convex-weakly-concave min-max problems for the training
of generative adversarial networks, the authors of \cite{JMLR:v22:20-533} proposed a weakly monotone condition for VIs. With the similar motivation, the authors of \cite{2020efficient} introduced a new class of structured nonconvex-nonconcave min-max optimization problems by proposing a weak  MVI (Minty variational inequality) condition, under which they proved sublinear convergence on the squared gradient norm for an EG-type method, named EG+ method. Another recent work \cite{2020simple} introduced a generalized strong monotonicity condition, with which the authors showed linear convergence for their proposed operator extrapolation method. It should be noted that all these mentioned works were done under the Lipschitz continuity assumption.

A natural question arises: can we combine the relative Lipschitzness and non-monotonicity to consider more general VIs? To answer this question, we first introduce two algorithmic frameworks--mirror EG method and mirror EP method by utilizing mirror mapping and extrapolation techniques as well as the tool of generalized Bregman distances. The new frameworks mainly depend on a mirror mapping function $\omega$. By suitably choosing subgradients from the subdifferential of $\omega$, we greatly expand our previously proposed Bregman EG and EP methods \cite{2021Extragradient}. When the mirror mapping function is specialized to the so-called $\cX$-regularizer, the mirror EG method recovers the unified mirror prox method in \cite{2019Unifying}. Formally, the mirror EG method is very similar to the original EG method, which helps us clearly distinguish Nemirovski's mirror prox method \cite{2004Prox}, Nesterov's dual extrapolation method \cite{2007dual}, and our Bregman EG method \cite{2021Extragradient}. The mirror EP method is formally similar to Bregman EP method and covers the operator EP method \cite{2020simple}. Theoretically, we derive a group of convergence results for our proposed frameworks under relative Lipschitzness and generalized monotone-like assumptions, which generalize and improve the existing theory. For example, Theorem \ref{thMEG2} for mirror EG method generalizes the convergence result of EG+ in \cite{2020efficient} and Theorem \ref{thMEP2} for mirror EP method generalizes and improves the convergence result of operator EP in \cite{2020simple}.

The remainder of the paper is organized as follows. In Section 2, we present some basic notations
and elementary preliminaries. In Section 3, the VI problem of interest and the saddle-point optimization problem are formulated, along with a group of Lipschitz-like and monotone-like assumptions. In Section 4, by introducing the mirror mapping and extrapolation techniques, we propose the mirror EG and EP methods. We also present an algorithmic application to a class of constrained saddle-point optimization problems. In Section 5, we study convergence properties  of the proposed frameworks under relative Lipschtizness and generalized monotone-type conditions. Section 6 gives some concluding remarks and several research directions for future work.

\section{Preliminaries}\label{se2}

\subsection{Notation}
In this paper, we restrict our attention to an arbitrary finite dimensional space $\RR^d$ associated with inner product $\langle \cdot, \cdot\rangle$. Denote the norm in $\RR^d$ by $\|\cdot\|$ and its dual norm by $\|\cdot\|_*$.
For a multi-variables function $f(x,y)$, we use $\nabla_xf$ (respectively, $\nabla_yf$) to denote the gradient of $f$ with respect to $x$ (respectively, $y$).
Given a subset $\Omega\subset\RR^d$, the characteristic function of $\Omega$ is denoted by $I_\Omega$ that takes value zero on $\Omega$ and equals to $+\infty$ elsewhere.

\subsection{Convex analysis tools}
We present some basic notations and facts about convex analysis, which will be used in our study.
\begin{definition}\label{sc0}
A function $\phi:\RR^d\rightarrow \RR\bigcup\{+\infty\}$ is convex if for any $\alpha\in [0,1]$ and $u, v\in \RR^d$, we have
\begin{equation*}
\phi(\alpha u+(1-\alpha)v)\leqslant \alpha \phi(u)+(1-\alpha)\phi(v);
\end{equation*}
and strongly convex with modulus $\mu> 0$ if for any $\alpha\in [0,1]$ and $u, v\in \RR^d$, we have
\begin{equation*}
\phi(\alpha u+(1-\alpha)v)\leqslant \alpha \phi(u)+(1-\alpha)\phi(v)-\frac{1}{2}\mu\alpha(1-\alpha)\|u-v\|^2. \label{SC1}
\end{equation*}
Furthermore, $\phi$ is concave if $-\phi$ is convex.

Let $f(x,y):\RR^{p+q}\rightarrow \RR\bigcup\{+\infty\}$ be a function. If $f(x,y)$ is convex in $x$ for any fixed $y$ and concave in $y$ for any fixed $x$, then we say it is convex-concave.
\end{definition}

\begin{definition}
Let $\phi:\RR^d\rightarrow \RR\bigcup \{+\infty\}$ be a convex function. The subdifferential of $\phi$  at $u\in \RR^d$ is defined as
$$\partial \phi (u):= \{ u^* \in \RR^d: \phi(v)\geqslant \phi(u)+ \langle u^*, v-u\rangle,\quad \forall v\in \RR^d \}.$$
The elements of $\partial \phi(u)$ are called the subgradients of $\phi$ at $u$. The effective domain of $\phi$ is defined as $\dom \phi :=\{x\in \RR^d: \phi(x)< +\infty\}$ and the effective domain of $\partial \phi$ is given by
$$\dom \partial\phi :=\{x\in \RR^d: \partial\phi(x)\neq \emptyset\}.$$
\end{definition}

The subdifferential generalizes the classical concept of differential because of the well-known fact that $\partial \phi(u)=\{\nabla \phi(u)\}$ when the function $\phi$ is differentiable.
In terms of the subdifferential, the strong convexity in Definition \ref{sc0} can be equivalently stated as \cite{Hiriart2004}: For any $u, v\in \RR^d$ and $v^*\in \partial\phi(v)$, we have
\begin{equation}\label{sc01}
 \phi(u)\geqslant \phi(v)+\langle v^*, u-v\rangle +\frac{\mu}{2}\|u-v\|^2.
\end{equation}

\begin{definition}
Let $\phi:\RR^d\rightarrow \RR\bigcup \{+\infty\}$ be a convex function. The conjugate of $\phi$ is defined as
$$\phi^*(u^*)=\sup_{v\in \RR^d}\{\langle u^*,v \rangle -\phi(v)\}.$$
\end{definition}

The following facts are well-known, which can be found from the classic textbooks \cite{1970convex} and \cite{Hiriart2004}.

\begin{lemma}\label{lemc}
Let $\phi:\RR^d\rightarrow \RR\bigcup \{+\infty\}$ be a lower-semicontinuous convex function with a nonempty domain. Then, the conditions $\phi(u)+\phi^*(u^*)=\langle u, u^*\rangle$, $u^*\in \partial \phi(u)$, and $u\in\partial \phi^*(u^*)$ are equivalent.
\end{lemma}

\subsection{Bregman distance tools}
The Bregman distance, originally introduced in \cite{1967The}, is a very powerful tool in many fields where distances are involved. Recently, a couple of variants of Bregman distances were introduced, see e.g. \cite{1997Legendre,1997Free,2018Re}. For simplicity as well as generality, we define the Bregman distance via a convex function, not necessarily differentiable.
\begin{definition}
Let $\omega:\RR^d\rightarrow \RR\bigcup \{+\infty\}$ be a convex function with $\dom\partial\omega\neq \emptyset$. The Bregman distance $D_\omega^{v^*}(u,v)$ between $u\in \RR^d$ and $v\in\dom\partial\omega$ with respect to $\omega$ and a subgradient $v^*\in\partial \omega(v)$ is defined by
\begin{equation}\label{Breg}
D_\omega^{v^*}(u,v):=\omega(u)-\omega(v)-\langle v^*, u-v\rangle.
\end{equation}
If the distance generating function $\omega$ is differentiable at $v$, the Bregman distance above reduces to
\begin{equation}\label{Breg1}
D_\omega(u,v):=\omega(u)-\omega(v)-\langle \nabla\omega(v), u-v\rangle.
\end{equation}
\end{definition}
The following results on the Bregman distance, essentially coming from \cite{1997Free,1997proximal}, will be used in the later convergence analysis.
%We list them here, along with a brief proof, for completeness.
\begin{lemma}\label{lemBreg}
Let $\omega:\RR^d\rightarrow \RR$ be a convex function with $\dom\partial\omega\neq \emptyset$. For any $u, p, q\in \RR^d$ and $p^*\in \partial \omega (p), q^*\in \partial \omega (q)$, we have that
\begin{equation}\label{Bregdis1}
D_\omega^{p^*}(u,p) - D_\omega^{q^*}(u,q) + D_\omega^{q^*}(p,q)=\langle q^*-p^*, u-p\rangle.
\end{equation}
If the function $\omega$ is $\mu$-strongly convex, then we have
\begin{equation}\label{Bregdis3}
D_\omega^{q^*}(p,q)\geqslant \frac{\mu}{2}\|p-q\|^2.
\end{equation}
\end{lemma}

\section{The problem formulation and assumptions}
In this section, we first introduce the VI problem of interest. Then, we discuss a group of assumptions on Lipschitz-like continuities and monotone-like conditions.

\subsection{The problem formulation}
Given a nonempty closed convex set $\cU\subset \RR^d$ and an operator $F:\cU\rightarrow \RR^d$, the VI problem of interest consists in finding $\hat{u}$ such that
\begin{equation}\label{VI}
  \langle F(\hat{u}), u-\hat{u}\rangle \geqslant 0, \forall u\in\cU.
\end{equation}
Such a point $\hat{u}$ is referred to as a strong solution of the VI problem \eqref{VI}. We denote the set of all such solutions by $\cU_s$, assumed to be nonempty. If the operator $F$ is monotone in the sense that
\begin{equation}\label{mono}
  \langle F(u)-F(u^\prime), u-u^\prime\rangle \geqslant 0, \forall u, u^\prime\in\cU,
\end{equation}
then the VI problem \eqref{VI} has the following equivalent problem that consists in finding
$\tilde{u}$ such that
\begin{equation}\label{VIw}
  \langle F(u), \tilde{u}-u\rangle \leqslant 0, \forall u\in\cU.
\end{equation}
Such a point $\tilde{u}$ is referred to as a weak solution of the VI problem \eqref{VI}. We denote the set of all such solutions by $\cU_w$, which equals to $\cU_s$ under the monotonicity of $F$.

As an illustrative example, we consider the following constrained saddle-point optimization problem, which consists in finding a point pair $(\hat{x},\hat{y})\in \cX\times \cY$ such that
\begin{equation}\label{sadd}
f(\hat{x},\hat{y})=\min_{x\in \cX}\max_{y\in \cY} f(x,y),
\end{equation}
where $\cX\subset\RR^p$ and $\cY\subset\RR^q$ are closed convex sets, and the saddle-point function $f:\cX\times\cY\rightarrow \RR$ is smooth convex-concave.
Such a point $(\hat{x},\hat{y})\in \cX\times \cY$ is referred to as a saddle point (also called a Nash equilibrium), which can be characterized by
$$f(\hat{x}, y)\leqslant f(\hat{x},\hat{y})\leqslant f(x,\hat{y}), \forall x\in \cX, y\in \cY.$$
Then, the constrained saddle-point problem can be equivalently reformulated into the VI problem \eqref{VI} with  $\cU=\cX\times \cY$, $u=(x,y), \hat{u}=(\hat{x},\hat{y})$, and the operator $F: \cU\rightarrow \RR^{p+q}$ being given by
\begin{equation}\label{sg}
F(u)=F(x,y):= (\nabla_xf(x, y),-\nabla_yf(x, y))^T.
\end{equation}
The operator above is called saddle-point gradient operator.

\subsection{Assumptions}
The most common assumption for studying VIs in the literature may be the $L$-Lipschitz continuity condition of the operator $F$, i.e., for some $L>0$,
\begin{equation}\label{Lip}
  \| F(u)-F(u^\prime)\|_*\leqslant L\|u-u^\prime\|, ~~\forall u, u^\prime\in\cU.
\end{equation}
If the operator $F$ is specialized to the gradient operator of a function $\phi$, then the property \eqref{Lip} recovers the gradient-Lipschitz-continuity condition (also called $L$-smoothness, see e.g. the book \cite{beck2017first}), which was recently generalized to so-called relative smoothness for wider first-order optimization in \cite{Bauschke2016A,Lu2016relatively}.
Motivated by the relative smoothness and the area convexity condition of \cite{2017area}, a weaker alternative to the $L$-Lipschitz continuity condition, called relative Lipschitzness, was introduced recently in \cite{2020Relative} with the form:
$$\langle F(v)-F(u),v-z\rangle \leqslant \lambda\cdot (D_\omega(v,u)+ D_\omega(z,v)),$$
where $\omega$ is a convex and differentiable function and $\lambda>0$ is the relatively Lipschitz constant.
The proposed relative Lipschitzness is well-suited for the standard analyses of extragradient-type methods, as illustrated in \cite{2020Relative} for mirror prox and dual extrapolation methods and in \cite{2021Extragradient} for Bregman extragradient and extrapolation methods. In this study, in order to apply the relative Lipschitzness more widely, we slightly modify it by allowing the distance generating function $\omega$ to be nondifferentiable.
\begin{assumption}\label{reLip}
Let $\omega:\RR^d\rightarrow \RR\bigcup\{+\infty\}$ be a convex function satisfying $\cU\subset \dom\partial\omega$ and let $\lambda$ be a positive parameter.  We say that the operator $F$ is $\lambda$-relatively Lipschitz continuous with respect to $\omega$ on $\cU$ if for any $u, v, z\in \cU$ it holds that
\begin{equation}\label{relip1}
\langle F(v)-F(u),v-z\rangle \leqslant \lambda\cdot (\inf_{u^*\in \partial \omega(u)}D_\omega^{u^*}(v,u)+ \inf_{v^*\in \partial \omega(v)}D_\omega^{v^*}(z,v)).
\end{equation}
\end{assumption}

The following two results, essentially observed in \cite{2020Relative}, respectively say that the relative Lipschitzness encapsulates the basic Lipschitz continuity and relative smoothness. Note that the relative notions in the second lemma are more general than the original statement in \cite{2020Relative} since we allow the distance generating function $\omega$ to be nondifferentiable. For completeness, we give proofs in Appendix.
\begin{lemma}\label{resclip}
If $\omega$ is strongly convex with modulus $\mu$ and $F$ is L-Lipschitz continuous, then $F$ is $L/\mu$-relatively Lipschitz continuous with respect to $\omega$ on the whole space $\RR^d$.
\end{lemma}

\begin{lemma}\label{relsmooth}
If $\phi$  is convex and $L$-relatively smooth with respect to $\omega$ in the senses that
\begin{equation}\label{relsmo}
 \phi(u)-\phi(v)-\langle \nabla \phi(v), u-v\rangle \leqslant L\cdot D_\omega^{v^*}(u,v), \forall u, v\in\RR^d,~v^*\in\partial \omega(v),
\end{equation}
then $F$, defined by $F:=\nabla \phi$, is $L$-relatively Lipschitz continuous with respect to $\omega$ on the whole space $\RR^d$.
\end{lemma}

The (relatively) Lipschitz continuity of $F$ defines (relatively) smooth VIs. In order to solve the VIs globally and more efficiently, monotone-like assumptions are usually necessary. Therefore, we introduce two monotone-like conditions, namely $\eta$-monotone and $\tau$-comonotone, as strengthen versions of \eqref{mono}.
The operator $F$ is $\eta$-monotone if for any $u, u^\prime \in\cU$ it holds that
\begin{equation}\label{paramono}
\langle F(u)-F(u^\prime),u- u^\prime \rangle\geqslant \eta\cdot \|u-u^\prime\|^2.
\end{equation}
The operator $F$ is $\tau$-comonotone if for any $u, u^\prime \in\cU$ it holds that
\begin{equation}\label{cmono}
\langle F(u)-F(u^\prime),u- u^\prime \rangle\geqslant \tau\cdot \|F(u)-F(u^\prime)\|_*^2.
\end{equation}
The $\eta$-monotonicity consists of three cases depending on the parameter: strong monotonicity when $\eta>0$, weak monotonicity when $\eta<0$, and monotonicity when $\eta=0$. Similarly, the $\tau$-comonotonicity also consists of three cases: cocoercivity when $\tau>0$, comonotonicity when $\tau<0$, and monotonicity when $\eta=0$; note that the definition of comonotonicity here is more general than that in \cite{bauschke2020generalized} since we distinguish the norm and dual norm.  There are interesting links between monotonicity of operator and convexity of function. A classic result is the correspondence between the convexity of a function $\phi$ and the monotonicity of its gradient $\nabla \phi$. Another well-known result is that the saddle-point gradient operator \eqref{sg} is monotone if the saddle-point function $f(x,y)$ is smooth convex-concave and vice versa. Based on these facts, we may guess that some classes of non-convex and non-concave saddle-point functions could be characterized by certain non-monotonicity conditions. Actually, the connection between the weak monotonicity of the saddle-point gradient operator \eqref{sg} and the weakly-convex-weakly-concave of the saddle-point function $f$ was recently established in \cite{JMLR:v22:20-533}. Thereby, further relaxed monotonicity conditions in VIs allow us to cover saddle-point optimization problems more widely.

Note that the monotone-like assumptions previously mentioned are all characterized by inequalities that require all point pairs $(u,u^\prime)$ to satisfy. This requirement is strict but not necessary because we actually only need these inequalities to hold on some \textsl{certain} point pairs. Hence, it leaves us space for further relaxation; this idea was developed in convex optimization for relaxing the strong convexity in \cite{2013Gradient,necoara2019linear}. Returning to VIs and with the same spirit of relaxation, the authors of \cite{2015On} introduced the generalized monotone assumption, which was rediscovered in the work \cite{2017stochastic} under the name of variational coherence.
In symbol, we say that the operator $F$ is generalized monotone if for any $\hat{u}\in\cU_s$ it holds that
\begin{equation}\label{gmono}
\langle F(u),u-\hat{u}\rangle \geqslant 0, \forall u\in \cU.
\end{equation}
Obviously, this condition is implied by the monotonicity and also by the pseudo-monotone condition, which says
that the operator $F$ is pseudo-monotone if for any $u, u^\prime \in\cU$ it holds that
\begin{equation}\label{pmono}
\langle F(u), u^\prime -u\rangle\geqslant 0\Rightarrow \langle F(u^\prime), u^\prime-u\rangle \geqslant 0.
\end{equation}
Conversely, the generalized monotonicity does not imply the monotonicity or the pseudo-monotonicity, and hence it is a strictly weaker assumption; see the constructed example in \cite{2015On}.
Recently, the authors of \cite{2020efficient} introduced the weak MVI condition, which is further weaker than the generalized monotonicity, for a class of structured nonconvex-nonconcave min-max optimization.
It says that the operator $F$ is $\rho$-weakly MVI monotone if $\rho>0$ and there exists $\hat{u}\in \cU_s$ such that for any $u \in\cU$ it holds that
\begin{equation}\label{wmono}
\langle F(u), u-\hat{u}\rangle\geqslant-\frac{\rho}{2}\|F(u)\|_*^2.
\end{equation}
It has been shown in \cite{2020efficient} that the $\rho$-weakly MVI monotone assumption is implied by the $(-\frac{\rho}{2})$-comonotonicity \cite{bauschke2020generalized} in the case of $\cU=\RR^d$. Based on the interaction dominate condition in \cite{grimmer2021landscape}, the authors of \cite{lee2021fast} introduced a nonconvex-nonconcave condition that implies the the $-\frac{\rho}{2}$-comonotonicity and hence implies the $\rho$-weakly MVI monotonicity.

In order to study linear convergence of the operator extrapolation method, the authors of \cite{2020simple} proposed a generalized strong monotonicity condition, which is stronger than the generalized monotonicity but does not imply monotonicity. Letting $\hat{u}\in \cU_s$, it reads as
\begin{equation}\label{gsmono}
\langle F(u), u-\hat{u} \rangle \geqslant 2\mu\cdot D_\omega(\hat{u},u), ~~\forall u \in\cU.
\end{equation}
When $F=\nabla \phi$ and $\omega=\frac{1}{2}\|\cdot\|^2$, the above inequality reduces to
\begin{equation}\label{rsc}
\langle \nabla\phi(u), u-\hat{u}\rangle \geqslant \mu\cdot\|\hat{u}-u\|^2, ~~\forall u \in\cU,
\end{equation}
which was introduced and studied in \cite{2013Aug,2013Gradient,zhang2017the} under the name of restricted strongly convex property and in \cite{2017stochastic} under the name of strongly variational coherence.
Below, we present a modified version of \eqref{gsmono} by allowing the distance generating function $\omega$ to be nondifferentiable, and call it relatively restricted monotonicity.
\begin{assumption}\label{reMono}
Let $\omega:\RR^d\rightarrow \RR\bigcup\{+\infty\}$ be a convex function with $\cU\subset \dom\partial\omega$ and let $\mu\geqslant 0$ be a given parameter.  We say that the operator $F$ is $\mu$-relatively restricted monotone with respect to $\omega$ on $\cU$ if for any $u\in\cU$ and $\hat{u}\in \cU_s$ it holds that
\begin{equation}\label{reMono1}
 \langle F(u), u-\hat{u} \rangle \geqslant \sup_{u^*\in \partial\omega(u)}2\mu\cdot D_\omega^{u^*}(\hat{u},u).
\end{equation}
\end{assumption}

\section{Algorithmic frameworks}\label{se3}
In this section, we introduce two mirror algorithmic frameworks, both of which are constructed by coupling the operator $F:\cU\rightarrow \RR^d$ and a convex function  $\omega:\RR^d\rightarrow \RR\bigcup\{+\infty\}$.

\subsection{Two pillars}
In this part, we introduce the mirror mapping and extrapolation techniques, as two basic tools for constructing our mirror frameworks.
\subsubsection{Mirror mapping technique}
The mirror descent (MD) algorithm, originally introduced by Nemirovski and Yudin in \cite{1983problem}, has become a very popular method in many applied mathematical fields. Here, we briefly recall its algorithmic motivation and the explanation in terms of nonlinear projection by following \cite{2003Mirror} and \cite{2014Convex}.

Suppose that we want to minimize a function $f$ in some Banach space $\cB$, whose dual space $\cB^*$ is assumed to be different from $\cB$. Then, the gradient descent strategy can not be directly applied to minimizing $f$ since now the formulation
$$x-\eta\cdot \nabla f(x)$$
makes nonsense if the step size $\eta\neq 0$ due to the fact that $x\in \cB$ whilst $\nabla f(x)\in \cB^*$. There are two possible ways to remedy this: mapping $x$ into $\cB^*$ or mapping $\nabla f(x)$ back to $\cB$ so that both of them lie in the same space. More concretely, we first choose a map $\Psi: \cB\rightarrow \cB^*$ and its inverse map $\Psi^{-1}: \cB^*\rightarrow \cB$. The first way goes like
$$\Psi^{-1}(\Psi(x)-\eta\cdot \nabla f(x)).$$
The second way works as
$$x-\eta\cdot \Psi^{-1}(\nabla f(x)).$$
The MD algorithm follows from the first way by setting $\Psi=\nabla \omega$ and $\Psi^{-1}=\nabla \omega^*$ with a differentiable function $\omega$; while the dual space preconditional gradient descent, recently introduced in \cite{2021Dual}, follows from the second way.

Now, we describe the basic step of the MD as follows; see also Figure \ref{fig:lazy-mirror-descent}.
\begin{equation}\label{md}
  x^+=\nabla \omega^*(\nabla \omega(x)-\eta \cdot\nabla f(x)).
\end{equation}
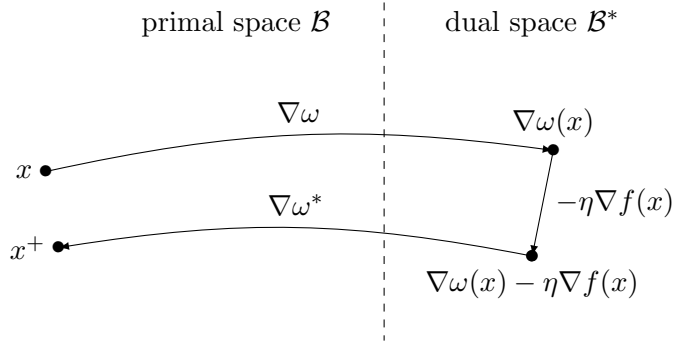
\begin{figure}
  \centering
\begin{tikzpicture}[y=0.80pt, x=0.80pt, yscale=-1.000000, xscale=1.000000, inner sep=0pt, outer sep=0pt,scale=2]
%  \path[draw=black,line join=miter,line cap=butt,even odd rule,line
%    width=0.800pt] (-24.9259,-39.9663) .. controls (-10.0099,-39.7633) and
%    (5.8309,-23.3238) .. (0.0000,0.0000) .. controls (-2.0000,8.0000) and
%    (-4.9107,16.3738) .. (-10.0000,22.3622) .. controls (-15.0893,28.3506) and
%    (-22.3573,31.9537) .. (-29.1922,30.4615);
\draw[->,>=latex] (-5,-10) node {$\bullet$} node[left=5pt] {$x$} to[bend right=10] node[midway,above=5pt]
{$\nabla \omega$} (115,-15);
\draw[->,>=latex] (115,-15) node {$\bullet$} node[above=5pt]{$\nabla \omega(x)$}
-- (110,10) node {$\bullet$}node {$\bullet$} node[below=5pt]{$\nabla \omega(x)-\eta\nabla f(x)$} node[midway,right=5pt]{$-\eta\nabla f(x)$};
\draw[->,>=latex] (110,10) to[bend left=10] node[midway,above=5pt]
{$\nabla \omega^*$}  (-2,8) node {$\bullet$} node[left=5pt] {$x^+$};

\draw[dashed] (75,-50) -- (75,30);
\draw (40,-45) node {primal space $\cB$};
\draw (110,-45) node {dual space $\cB^*$};
%\draw (-20,20) node {$\cB$};
%\draw (100,20) node {$\cB^*$};
\end{tikzpicture}

  \caption{Mirror Descent}
  \label{fig:lazy-mirror-descent}
\end{figure}
Note that if $\omega$ is strongly convex, then \eqref{md} is just the first-order optimality condition of the following convex optimization:
\begin{eqnarray}\label{nonp}
\begin{array}{lll}
x^+ &= & \arg\min_z\{ \langle \eta \nabla f(x), z\rangle +\omega(z)- \langle \nabla \omega(x), z\rangle \} \\
  &= & \arg\min_z\{ \langle \eta \nabla f(x), z\rangle +D_\omega(z,x) \}.
\end{array}
\end{eqnarray}
This is exactly the nonlinear projection explanation introduced in \cite{2003Mirror}. If let $\omega=\frac{1}{2}\|\cdot\|_2^2$ with $\|\cdot\|_2$ being the Euclidean norm, then we have
\begin{equation}\label{gd}
  x^+=x-\eta\nabla f(x)=\arg\min_z\{ \langle \eta \nabla f(x), z\rangle +\frac{1}{2}\|z-x\|_2^2 \}.
\end{equation}
Therefore, \eqref{nonp} can be viewed as a generalized gradient descent (also called Bregman gradient descent) since it replaces the Euclidean distance proximity term $\frac{1}{2}\|z-x\|_2^2$ with the Bregman distance $D_\omega(z,x)$.
In order to further generalize \eqref{nonp}, we consider to replace the term $D_\omega(z,x)$ by the generalized Bregman distance $D_\omega^{x^*}(z,x)$ with $x^*\in \partial \omega(x)$; that is
\begin{eqnarray}\label{gmd}
\begin{array}{lll}
x^+ &= & \arg\min_z\{ \langle \eta \nabla f(x), z\rangle +D^{x^*}_\omega(z,x) \} \\
  &= & \nabla \omega^*(x^*-\eta \nabla f(x)).
\end{array}
\end{eqnarray}
We call the update \eqref{gmd} generalized mirror descent method; see also Figure \ref{fig:g-mirror-descent}. Note that this update not only depends on the function $\omega$ but also the subgradient $x^*$. Different setting rules of such $x^*$ will correspond to different algorithmic schemes.

\subsubsection{Extrapolation technique}
The extragradient (EG) method, originally introduced by Korpelevich in 1976 \cite{1976An}, is also a modified gradient method using the idea of extrapolation rather than the mirror mapping. It first makes a trial step along a negative gradient direction to produce an ``extrapolated" point, and then proceeds actual movement along the gradient at the extrapolated point. Applying this idea to the VI problem \eqref{VI} with $\cU=\RR^d$, the basic step of the EG method can be described as follows:
\begin{eqnarray}\label{EG}
\left\{\begin{array}{lll}
\bar{u}_k =  u_k- \eta F(u_k),   \\
u_{k+1} = u_k-\eta F(\bar{u}_k).
\end{array} \right.
\end{eqnarray}
Here, $\bar{u}_k$ is the extrapolated point and $u_{k+1}$ is the actually updated point from $u_k$. In \eqref{EG}, two gradient-like sequences $\{F(u_k)\}$ and $\{F(\bar{u}_k)\}$ need to compute. Recently, a single-call EG variant, which only requires to compute the sequence $\{F(\bar{u}_k)\}$, received lots of attention \cite{2019onthe,2021linear,2021fast,2021the}. This variant reads as
\begin{eqnarray}\label{EGv1}
\left\{\begin{array}{lll}
\bar{u}_k =  u_k- \eta F(\bar{u}_{k-1}),   \\
u_{k+1} = u_k-\eta F(\bar{u}_k).
\end{array} \right.
\end{eqnarray}
One can see that the new extrapolated point $\bar{u}_k$ is obtained along the history gradient $F(\bar{u}_{k-1})$ rather than $F(u_k)$. This slight modification provides us with the following equivalent form of \eqref{EGv1}, that is
\begin{equation}\label{EGv}
\bar{u}_{k+1}=\bar{u}_k-\eta F(\bar{u}_k)-\eta(F(\bar{u}_k)-F(\bar{u}_{k-1})).
\end{equation}
The term $F(\bar{u}_k)-F(\bar{u}_{k-1})$ is referred to as an operator extrapolated term. At last, it should be noted that \eqref{EGv} can also be derived from Popov's modification \cite{1980A}, which reads as
\begin{eqnarray}\label{EGv2}
\left\{\begin{array}{lll}
\bar{u}_k =  \bar{u}_{k-1}- \eta F(u_k),   \\
u_{k+1} = \bar{u}_k-\eta F(u_k).
\end{array} \right.
\end{eqnarray}
Eliminating the element $\bar{u}_k$, we can get
\begin{equation}\label{EGve}
u_{k+1}=u_k-\eta F(u_k)-\eta(F(u_k)-F(u_{k-1})),
\end{equation}
which is essentially \eqref{EGv}. Note that Popov's method is also a single-call EG-type variant.

\begin{figure}
  \centering
\begin{tikzpicture}[y=0.80pt, x=0.80pt, yscale=-1.000000, xscale=1.000000, inner sep=0pt, outer sep=0pt,scale=2]
%  \path[draw=black,line join=miter,line cap=butt,even odd rule,line
%    width=0.800pt] (-24.9259,-39.9663) .. controls (-10.0099,-39.7633) and
%    (5.8309,-23.3238) .. (0.0000,0.0000) .. controls (-2.0000,8.0000) and
%    (-4.9107,16.3738) .. (-10.0000,22.3622) .. controls (-15.0893,28.3506) and
%    (-22.3573,31.9537) .. (-29.1922,30.4615);
\draw[->,>=latex] (-5,-10) node {$\bullet$} node[left=5pt] {$x$} to[bend right=10] node[midway,above=5pt]
{$\partial \omega$} (115,-15);
\draw[->,>=latex] (115,-15) node {$\bullet$} node[right=5pt]{$x^*\in \partial \omega(x)$}
-- (110,10) node {$\bullet$}node {$\bullet$} node[below=5pt]{$x^*-\eta\nabla f(x)$} node[midway,right=5pt]{$-\eta\nabla f(x)$};
\draw[->,>=latex] (110,10) to[bend left=10] node[midway,above=5pt]
{$\nabla \omega^*$}  (-2,8) node {$\bullet$} node[left=5pt] {$x^+$};

\draw[dashed] (75,-50) -- (75,30);
\draw (40,-45) node {primal space $\cB$};
\draw (110,-45) node {dual space $\cB^*$};
%\draw (-20,20) node {$\cB$};
%\draw (100,20) node {$\cB^*$};
\end{tikzpicture}
  \caption{Generalized Mirror Descent}
  \label{fig:g-mirror-descent}
\end{figure}
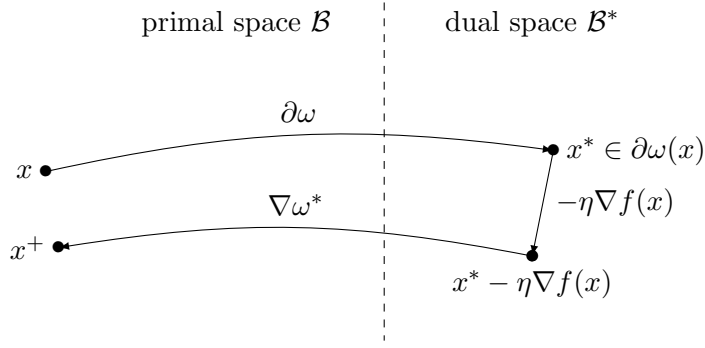

\subsection{The proposed methods}

Very recently, we introduced two new EG variants, namely Bregman EG and EP methods in \cite{2021Extragradient}, by applying the generalized mirror descent \eqref{gmd} to the EG method \eqref{EG} and to its extrapolated variant \eqref{EGve}, respectively. Now, let us recall these two methods.
The Bregman EG method generates the iterates $\{u_k\}$ for $k\geqslant0$ via the following scheme:
\begin{eqnarray}\label{BEG}
\left\{\begin{array}{lll}
\bar{u}_k =  \nabla \omega^*(u_k^*- \alpha_k F(u_k)),   \\
u_{k+1} =  \nabla \omega^*(u_k^*- \alpha_k F(\bar{u}_k)),    \\
u_{k+1}^*= u_k^*- \alpha_k F(\bar{u}_k).
\end{array} \right.
\end{eqnarray}
The Bregman EP method generates the iterates $\{u_k\}$ for $k\geqslant0$ via the following scheme:
\begin{eqnarray}\label{BEP}
\left\{\begin{array}{lll}
u_{k+1}=  \nabla \omega^*(u_k^*- \alpha_k F(u_k)-\alpha_k\beta_k(F(u_k)-F(u_{k-1}))),   \\
u_{k+1}^*= u_k^*- \alpha_k F(u_k)-\alpha_k\beta_k(F(u_k)-F(u_{k-1})).
\end{array} \right.
\end{eqnarray}
In this study, we further generalize the Bregman EG and EP methods to solve VIs by utilizing other possible subgradients from the subdifferential $\partial\omega(u_k)$. Corresponding to the Bregman EG method, we propose the following scheme, called mirror EG method.
\begin{subequations}\label{MEG}
\begin{align}[left = \empheqlbrace\,]
\label{MEGa}&\bar{u}_k =  \nabla \omega^*(u_{k,1}^*- \frac{\alpha_k}{\beta} F(u_k)), u^*_{k,1}\in \partial\omega(u_k),  \\
\label{MEGb}&u_{k+1} =  \nabla \omega^*(u_{k,2}^*- \alpha_k F(\bar{u}_k)), u^*_{k,2}\in \partial\omega(u_k),   \\
\label{MEGc}&\langle u_{k,2}^*-u_{k,1}^*, u-u_k\rangle \leqslant 0, \forall u\in \cU,\\
\label{MEGd}&\langle u_{k+1,2}^*-u_{k,2}^*+ \alpha_k F(\bar{u}_k), u-u_{k+1}\rangle \geqslant 0, \forall u\in \cU.
\end{align}
\end{subequations}
In \eqref{MEGa} and \eqref{MEGb}, we utilize two different subgradient sequences $\{u_{k,1}^*\}$ and $\{u_{k,2}^*\}$, which distinguishes the mirror EG method from the Bregman EG method. Of course, the choosing of subgradients can not be arbitrary. In \eqref{MEGc}, we add a constraint on them, in order to ensure the convergence analysis of the Bregman EG in \cite{2021Extragradient} still to be feasible. The reader who will go through our forthcoming analysis could find that it is sufficient to require the weaker constraint $ \langle u_{k,2}^*-u_{k,1}^*, u_{k+1}-u_k\rangle \leqslant 0$, or equivalently,
$$D_\omega^{u^*_{k,1}}(u_{k+1}, u_k)-D_\omega^{u^*_{k,2}}(u_{k+1}, u_k)\leqslant 0.$$
From \eqref{MEGd}, we have $\langle \alpha_k F(\bar{u}_k), u_{k+1}-u\rangle\leqslant \langle u_{k+1,2}^*-u_{k,2}^*, u-u_{k+1}\rangle$, which actually relaxes the equality
$$\langle \alpha_k F(\bar{u}_k), u_{k+1}-u\rangle= \langle u_{k+1,2}^*-u_{k,2}^*, u-u_{k+1}\rangle.$$
This relaxation does not affect the estimation of  $\langle \alpha_k F(\bar{u}_k), u_{k+1}-u\rangle$ by  $\langle u_{k+1,2}^*-u_{k,2}^*, u-u_{k+1}\rangle.$
At last, the introduction of parameter $\beta$ was inspired by the EG+ method, recently suggested in \cite{2020efficient}. The reader will find that if $\beta$ is strictly less than one, then sublinear convergence of the mirror EG method may follow under the weakly MVI monotonicity.

To summarize, our motivation of proposing the mirror EG method is to utilize the subdifferential $\partial \omega(u_k)$  more sufficiently and meanwhile guarantee it to converge safely.

Corresponding to the Bregman EP method, we similarly propose the mirror EP method as follows.
\begin{subequations}\label{MEP}
\begin{align}[left = \empheqlbrace\,]
\label{MEPa}&\xi_k=\alpha_k F(u_k)+\alpha_k\beta_k(F(u_k)-F(u_{k-1})),\\
\label{MEPb}&u_{k+1}=   \nabla \omega^*(u_k^*- \xi_k ), u^*_k\in \partial \omega(u_k),   \\
\label{MEPc}&\langle u_{k+1}^*- u_k^*+ \xi_k, u-u_{k+1}\rangle \geqslant 0, \forall u\in\cU.
\end{align}
\end{subequations}

In order to guarantee that the sequences $\{\bar{u}_k\}$ and $\{u_k\}$ in \eqref{MEG} and \eqref{MEP} are well-defined, we need to choose suitable functions $\omega$, called mirror mapping function, such that: 1) the subdifferential $\partial\omega (u_k)$  should be nonempty, otherwise $u^*_{k,i}$ can not be defined; 2) the conjugate $\omega^*$ should be differentiable on $\RR^d$ and the generated points $\bar{u}_k$ and $u_k$ by the mapping $\nabla \omega^*$ need to lie in $\cU$, otherwise $F(\bar{u}_k)$ and $F(u_k)$ would make nonsense since the operator $F$ is defined on $\cU$. To this end, we propose the following condition:
\begin{equation}\label{condw}
u_0\in \dom \partial\omega, ~~ \nabla\omega^*(\RR^d)\subset \cU,
\end{equation}
with which we immediately have $\bar{u}_k,u_k\in \nabla\omega^*(\RR^d)\subset \cU$ and hence $F(\bar{u}_k)$ and $F(u_k)$ are well-defined. Let us check that $\partial \omega(u_k)\neq \emptyset$. It suffices to show $\nabla\omega^*(\RR^d)\subset \dom \partial \omega$ since $u_k\in \nabla\omega^*(\RR^d)$. In fact, for any $v\in \nabla\omega^*(\RR^d)$, there exists $y\in \RR^d$ such that $v=\nabla\omega^*(y)$ and hence $y\in\partial \omega(v)$, which means $v\in \dom\partial \omega$. Thus, $\nabla\omega^*(\RR^d)\subset \dom \partial \omega$.

In the case of $\cU=\RR^d$, \eqref{condw} holds trivially. In the case of $\cU\neq \RR^d$, we may take $\omega$ to be of the form $\psi+I_\cU$ with some convex function $\psi$. Then, \eqref{condw} also holds because of
$$\nabla\omega^*(\RR^d)\subset \dom \omega\subset \cU,$$
where the left-hand side inclusion is due to the following observation
\begin{eqnarray}\label{eqeq}
\begin{array}{lll}
 v=\nabla\omega^*(y) &\Rightarrow& y\in \partial\omega(v) \\
&\Rightarrow & v\in \arg\min_{x\in \RR^d}\{\omega(x)-\langle y,x\rangle\} \\
&\Rightarrow & v\in \dom\omega.
\end{array}
\end{eqnarray}
For certain setups on the function $\psi$, the reader may refer to Section \ref{subsec:al}.

\subsection{Related unified framework}
Recently, the authors of \cite{2019Unifying} introduced and analyzed a unified mirror descent method which unifies both the mirror descent and dual averaging algorithms. Moreover, they applied it to solving Lipschitz continuous and monotone VI problems and proposed the unified mirror prox (UMP) method. Their methods and convergence analysis depend on a multi-valued prox-mapping $\Pi_\omega:\RR^d\rightrightarrows\cU\times \RR^d$. Concretely, $\Pi_\omega(\zeta)$ is the set of couples $(u, \upsilon)$ satisfying
\begin{align*}[left = \empheqlbrace\,]
&u=\nabla \omega^*(\zeta),\\
&\upsilon\in \partial \omega(u),\\
&\langle \upsilon-\zeta, u^\prime-u\rangle\geqslant 0, \forall u^\prime\in \cU.
\end{align*}
With this operator, their UMP can be described as follows:
\begin{subequations}
\begin{align}[left = \empheqlbrace\,]
\label{eq:ump-zeta}&\zeta_k\in \partial \omega(u_k),\\
\label{eq:ump-var}&\langle \upsilon_k-\zeta_k, u-u_k\rangle\leqslant 0, \forall u\in \cU,\\
\label{eq:ump-y}&\bar{u}_k=\nabla\omega^*(\zeta_k-\gamma F(u_k)),\\
\label{eq:ump-x}&(u_{k+1}, \upsilon_{k+1})\in \Pi_\omega(\upsilon_k-\gamma F(\bar{u}_k)).
\end{align}
\end{subequations}
It was verified that both of Nemirovski's mirror prox method \cite{2004Prox} and Nesterov's dual extrapolation \cite{2007dual} method are special cases of UMP. Interestingly, we find when the mirror mapping function $\omega$ is specialized to the $\cX$-regularizer defined in \cite{2019Unifying}, UMP can be transformed into the scheme \eqref{MEG}. In other words, they are equivalent in this special setting. This sounds surprising since they are designed with different motivations and hence have quite different appearances. We believe that each of them complements to the other.
On one hand, the existing theory on UMP will help us exploit the mirror EG method more deeply; on the other hand,
the reader may be more familiar with the formulation of the mirror EG method since it is very similar to the original form of the EG method. The latter will be highlighted by reformulating Nemirovski's mirror prox method, Nesterov's dual extrapolation method, and the Bregman EG method as special cases of the mirror EG method.

\subsection{Specialization}\label{subsec:al}
Abstract algorithmic frameworks have been constructed in the previously subsections. Now, we try to specialize the mirror mapping function $\omega$ so that concrete algorithmic examples could be discussed. First, we need the concept of compatible mirror map and its associated properties \cite{2019Unifying}.
\begin{definition}
  \label{def:mirror-maps}
  Let $\psi:\RR^d\to \RR\cup\{+\infty\}$ be a function.
Denote $\mathcal{D}_\psi:=\inte\dom\psi$, where the symbol $\inte\dom$ means the interior of $\dom\psi$.
We say that $\psi$ is a $\cU$-compatible \emph{mirror map} if
\begin{enumerate}
\item\label{item:F-lsc-strict-convex} $\psi$ is lower-semicontinuous and strictly convex,
\item\label{item:F-diff} $\psi$ is differentiable on $\mathcal{D}_\psi$,
\item\label{item:gradient-all-values} the gradient of $\psi$ takes all possible values, i.e.\ $\nabla \psi(\mathcal{D}_\psi)=\RR^d$.
\item\label{item:X-subset-bar-D} $\cU\subset \textrm{cl}(\mathcal{D}_\psi)$,
\item\label{item:X-inter-D-nonempty} $\cU\cap \mathcal{D}_\psi\neq \varnothing$.
\end{enumerate}
\end{definition}

\begin{lemma}\label{lemmirr}
Let $\psi:\RR^d\to \RR\cup\{+\infty\}$ be a $\cU$-compatible \emph{mirror map}. Then, we have
\begin{enumerate}
\item[(a).] $\nabla \psi^*(\RR^d)=\inte\dom\psi$,
\item[(b).] $\nabla \psi(\nabla\psi^*(v))=v, \forall v\in \RR^d$,
\item[(c).] if $z\in \inte\dom\psi$, then $\arg\min_{u\in\cU}D_\psi(u,z)$ exists and is unique, belonging to $\inte\dom\psi\bigcap\cU$,
\item[(d).] if $\omega=\psi+I_\cU$, then $\omega^*$ is differentiable on $\RR^d$ and $\nabla \psi(u)\in \partial \omega(u), \forall u\in \inte\dom\psi$.
\end{enumerate}
\end{lemma}

The definition of compatible mirror map and the listed properties above help us reformulate Nemirovski's mirror prox method, Nesterov's dual extrapolation method, and the operator extrapolation method.

\begin{proposition}\label{Nemi}
Let $\psi:\RR^d\to \RR\cup\{+\infty\}$ be a  $\cU$-compatible \emph{mirror map} and let $u_0\in \inte\dom\psi\bigcap\cU$. Describe the basic steps of Nemirovski's mirror prox method as follows:
 \begin{subequations}\label{Nemi0}
\begin{align}[left = \empheqlbrace\,]
\label{Nemi01}&\bar{u}_k=\arg\min_{u\in \cU}\{\langle \alpha_kF(u_k),u\rangle + D_\psi(u,u_k)\}\\
\label{Nemi02}&u_{k+1}=\arg\min_{u\in \cU}\{\langle \alpha_kF(\bar{u}_k),u\rangle + D_\psi(u,u_k)\}.
\end{align}
  \end{subequations}
Let $\omega:=\psi+I_\cU$. Then Nemirovski's mirror prox method can be equivalently described as
\begin{subequations}\label{Nemi1}
\begin{align}[left = \empheqlbrace\,]
\label{Nemi11}&\bar{u}_k =  \nabla \omega^*(\nabla\psi(u_k)- \alpha_k F(u_k)),  \\
\label{Nemi12}&u_{k+1} =  \nabla \omega^*(\nabla\psi(u_k)- \alpha_k F(\bar{u}_k)).
\end{align}
\end{subequations}
Moreover, it is a special case of the mirror EG method with $u^*_{k,i}=\nabla\psi(u_k), i=1,2.$
\end{proposition}

\begin{proof}
 First, we assume that $u_k\in \inte\dom\psi\bigcap\cU$. Starting with \eqref{Nemi01}, we derive that
\begin{eqnarray}\label{b1}
\begin{array}{lll}
 \bar{u}_k &=& \arg\min_{u\in \cU}\{ \langle \alpha_kF(u_k),u\rangle +\psi(u)-\langle \nabla \psi(u_k),u\rangle\} \\
&= & \arg\min_{u\in \cU}\{ \psi(u)-\langle \nabla\psi(\nabla \psi^*(\nabla \psi(u_k)-\alpha_kF(u_k))),u\rangle\} \\
&= & \arg\min_{u\in \cU}\{ D_\psi(u, \nabla \psi^*(\nabla \psi(u_k)-\alpha_kF(u_k)))\},
\end{array}
\end{eqnarray}
where the second relationship follows from (b) of Lemma \ref{lemmirr}. Using (a) of Lemma \ref{lemmirr}, we have
$$\nabla \psi^*(\nabla \psi(u_k)-\alpha_kF(u_k))\in \inte\dom\psi.$$
Thus, from (c) of Lemma \ref{lemmirr}, it follows that $\bar{u}_k\in \inte\dom\psi\bigcap\cU$. Similarly, we have $u_{k+1}\in \inte\dom\psi\bigcap\cU$. By induction and the assumption $u_0\in \inte\dom\psi\bigcap\cU$, we actually have shown that $u_k\subset  \inte\dom\psi\bigcap\cU$. Hence, $\{\nabla \psi(u_k)\}$ are well-defined. Moreover, from (d) of Lemma \ref{lemmirr}, we have $\nabla \psi(u_k)\in \partial \omega(u_k)$ since $u_k\in \inte\dom\psi$.

Using the first-order optimality condition of \eqref{Nemi01}, we obtain
$$0\in \partial \omega(\bar{u}_k)+\alpha_k F(u_k)-\nabla \psi(u_k),$$
which can be equivalently written as \eqref{Nemi11} by using Lemma \ref{lemc}. Similarly, \eqref{Nemi02} is equivalent to \eqref{Nemi12}. The condition $u_0\in \inte\dom\psi\bigcap\cU$ and the form $\omega=\psi+I_u$ ensure \eqref{eqeq} to hold, partially due to $\inte\dom \psi\subset \nabla\omega^*(\RR^d)$ from (d) of Lemma \ref{lemmirr}.
To show  \eqref{Nemi0} is a special case of the mirror EG method, it remains to show that
$$\langle \nabla \psi(u_{k+1})-\nabla \psi(u_k)+\alpha_kF(\bar{u}_k), u-u_{k+1}\rangle\geqslant 0, \forall u\in \cU,$$
which is just the first-order optimality condition of \eqref{Nemi02}. This completes the proof.
\end{proof}

\begin{proposition}\label{Nest}
Let $\psi:\RR^d\to \RR\cup\{+\infty\}$ be $\cU$-compatible \emph{mirror map} and let $u_0\in \inte\dom\psi\bigcap\cU$. Denote $\omega:=\psi+I_\cU$.  Describe the basic steps of Nesterov's dual extrapolation method as follows:
 \begin{subequations}\label{Nest1}
\begin{align}[left = \empheqlbrace\,]
\label{Nest11}&u_k=\nabla \omega^*(\upsilon_k),\\
\label{Nest12}&\bar{u}_k=\arg\min_{u\in \cU}\{\langle \alpha_kF(u_k),u\rangle + D_\psi(u,u_k)\}\\
\label{Nest13}&\upsilon_{k+1}=\upsilon_k-\alpha_kF(\bar{u}_k).
\end{align}
  \end{subequations}
Then, Nesterov's dual extrapolation method can be equivalently described as
\begin{subequations}\label{Nest2}
\begin{align}[left = \empheqlbrace\,]
\label{Nest21}&\bar{u}_k =  \nabla \omega^*(\nabla\psi(u_k)- \alpha_k F(u_k)),  \\
\label{Nest22}&u_{k+1} =  \nabla \omega^*(\upsilon_k- \alpha_k F(\bar{u}_k)),\\
\label{Nest23}&\upsilon_{k+1}=\upsilon_k-\alpha_kF(\bar{u}_k).
\end{align}
\end{subequations}
Moreover, it is a special case of the mirror EG method with $u^*_{k,1}=\nabla\psi(u_k)$ and $u^*_{k,2}=\upsilon_k$.
\end{proposition}

\begin{proof}
Repeating the argument in the proof of Proposition \ref{Nemi}, we deduce \eqref{Nest21} from \eqref{Nest12}. Using \eqref{Nest11} and \eqref{Nest13}, we deduce that $u_{k+1}= \nabla \omega^*(\upsilon_{k+1})=\nabla \omega^*(\upsilon_k- \alpha_k F(\bar{u}_k))$, which is just \eqref{Nest22}. Thus, the equivalence between \eqref{Nest1} and \eqref{Nest2} follows. To show that both of them are special cases of the mirror EG method, it remains to show that
$$\langle \upsilon_k-\nabla \psi(u_k), u-u_k\rangle \leqslant 0, \forall u\in\cU,$$
which follows by observing that $u_k=\nabla \omega^*(\upsilon_k)$ implies
$$\upsilon_k\in \partial \omega(u_k)=\nabla\psi(u_k)+N_\cU(u_k),$$
where $N_\cU(u_k)=\{y: \langle y, u-u_k\rangle \leqslant 0, \forall u\in \cU\}.$ This completes the proof.
\end{proof}

\begin{proposition}\label{OE}
Let $\psi:\RR^d\to \RR\cup\{+\infty\}$ be a $\cU$-compatible \emph{mirror map} and let $u_0\in \inte\dom\psi\bigcap\cU$. Denote $\omega:=\psi+I_\cU$.  Describe the basic steps of operator extrapolation method as follows:
\begin{equation}\label{OE1}
u_{k+1}= \arg\min_{u\in \cU}\{\langle \alpha_k F(u_k)+\alpha_k\beta_k(F(u_k)-F(u_{k-1})),u\rangle + D_\psi(u,u_k)\}.
\end{equation}
Then,  the operator extrapolation method can be equivalently described as
\begin{equation}\label{OE2}
u_{k+1}=  \nabla \omega^*(\nabla\psi(u_k)- \alpha_k F(u_k)-\alpha_k\beta_k(F(u_k)-F(u_{k-1}))).
\end{equation}
Moreover, it is a special case of the mirror EP method with $u^*_k=\nabla\psi(u_k)$.
\end{proposition}
The proof is similar to that of Proposition \ref{Nest}; we omit the details here.
Comparing the equivalent reformulations \eqref{Nemi1} and \eqref{Nest2} with the Bregman EG method \eqref{BEG}, we find that the latter is indeed a new algorithm in the sense it does not need to compute $\nabla\psi(u_k)$ at all. Also, the Bregman EP method \eqref{BEP} is new, compared with \eqref{OE1}, in the same sense.

\subsection{Examplary applications}
In this part, we apply our algorithmic frameworks to the constrained saddle point optimization problem \eqref{sadd} with the constraints given by
$$\cX=\Delta_p, \cY=\Delta_q,$$
where we use the symbol $\Delta_d=\{u\in \RR^d_+: \sum_{i=1}^d u_i=1\}$ to denote the $(d-1)$-dimensional probability simplex.
Thus, the constraint set in the corresponding VI problem \eqref{VI} is $\cU=\cX\times \cY$. We choose $\omega=\psi+I_\cU$ with $\psi(u)=\sum u_i\ln u_i-u_i$, whose conjugate function and gradients can be computed as follows:
$$\psi^*(v)=\sum_i \exp(v_i), \nabla\psi(u)=\ln(u), \nabla\psi^*(v)=\exp(v).$$
Note that the dimensions of variables in the above functions are not indexed so that the functions could be applied to different dimensions of variables.

In our algorithmic frameworks, the most important ingredient is to compute $\nabla \omega^*$. To do this, we deduce the computing $u=\nabla\omega^*(v)$ with $v\in \RR^{p+q}$ being given as follows:
\begin{eqnarray}\label{computu}
\begin{array}{lll}
  u &=& \arg\min_z\{\omega(z)-\langle v,z\rangle \}\\
&= & \arg\min_{z\in\cU}\{\psi(z)-\langle v,z\rangle\}.
\end{array}
\end{eqnarray}
Denote $z=(z_x,z_y)$ and $v=(v_x,v_y)$, where $z_x$ means the subvector of $z$ that lies in $\cX$. Using the separability of the objective function and the variables $x, y$. The computing \eqref{computu} can be transformed into the following two problems:
\begin{subequations}
\begin{align}
\label{compx}x= \arg\min_{z_x\in\cX} \{\psi(z_x)-\langle v_x,z_x\rangle\},\\
\label{compy}y = \arg\min_{z_y\in\cY}\{\psi(z_y)-\langle v_y,z_y\rangle\}.
\end{align}
\end{subequations}
Both of them can be expressed as Bregman's projection onto the constrained sets, since
\begin{eqnarray}\label{compx1}
\begin{array}{lll}
x &= \arg\min_{z_x\in\cX} \{\psi(z_x)-\langle \nabla\psi(\nabla\psi^*(v_x)),z_x\rangle\}\\
&=\arg\min_{z_x\in\cX} \{D_\psi(z_x,\nabla\psi(\nabla\psi^*(v_x)))\}.
\end{array}
\end{eqnarray}
Thus, using the formula in Example 6.15 in \cite{1997Legendre}, we obtain
\begin{equation}\label{xyf}
x=\frac{\nabla\psi(\nabla\psi^*(v_x))}{\sum_i (\nabla\psi(\nabla\psi^*(v_x)))_i}, ~~y=\frac{\nabla\psi(\nabla\psi^*(v_y))}{\sum_i (\nabla\psi(\nabla\psi^*(v_y)))_i}.
\end{equation}
Now, we are ready to present the following algorithmic schemes for the constrained saddle point problem. The first is obtained by applying Nemirovski's mirror prox method \eqref{Nemi1}:
\begin{subequations}\label{A1}
\begin{align}[left = \empheqlbrace\,]
&\bar{x}_k=\frac{\exp(\ln x_k-\alpha_k\nabla_xf(x_k,y_k))}{\sum_i (\exp(\ln x_k-\alpha_k\nabla_xf(x_k,y_k)))_i}, ~\bar{y}_k=\frac{\exp(\ln y_k-\alpha_k\nabla_yf(x_k,y_k))}{\sum_i (\exp(\ln y_k-\alpha_k\nabla_yf(x_k,y_k)))_i},  \\
&x_{k+1}=\frac{\exp(\ln x_k-\alpha_k\nabla_xf(\bar{x}_k,\bar{y}_k))}{\sum_i (\exp(\ln x_k-\alpha_k\nabla_xf(\bar{x}_k,\bar{y}_k)))_i}, ~y_{k+1}=\frac{\exp(\ln y_k-\alpha_k\nabla_yf(\bar{x}_k,\bar{y}_k))}{\sum_i (\exp(\ln y_k-\alpha_k\nabla_yf(\bar{x}_k,\bar{y}_k)))_i}.
\end{align}
\end{subequations}
The second is obtained by applying Nesterov's dual extrapolation method \eqref{Nest2}:
\begin{subequations}\label{A2}
\begin{align}[left = \empheqlbrace\,]
&\bar{x}_k=\frac{\exp(\ln x_k-\alpha_k\nabla_xf(x_k,y_k))}{\sum_i (\exp(\ln x_k-\alpha_k\nabla_xf(x_k,y_k)))_i}, ~\bar{y}_k=\frac{\exp(\ln y_k-\alpha_k\nabla_yf(x_k,y_k))}{\sum_i (\exp(\ln y_k-\alpha_k\nabla_yf(x_k,y_k)))_i},  \\
&x_{k+1}=\frac{\exp(x_k^*-\alpha_k\nabla_xf(\bar{x}_k,\bar{y}_k))}{\sum_i (\exp(x_k^*-\alpha_k\nabla_xf(\bar{x}_k,\bar{y}_k)))_i}, ~y_{k+1}=\frac{\exp(y_k^*-\alpha_k\nabla_yf(\bar{x}_k,\bar{y}_k))}{\sum_i (\exp(y_k^*-\alpha_k\nabla_yf(\bar{x}_k,\bar{y}_k)))_i},\\
&(x^*_{k+1},y^*_{k+1})=(x^*_k,y^*_k)-\alpha_k(\nabla_xf(\bar{x}_k,\bar{y}_k), -\nabla_yf(\bar{x}_k,\bar{y}_k)).
\end{align}
\end{subequations}
The third is obtained by applying Bregman EG method \eqref{BEG}:
\begin{subequations}\label{A3}
\begin{align}[left = \empheqlbrace\,]
&\bar{x}_k=\frac{\exp(x_k^*-\alpha_k\nabla_xf(x_k,y_k))}{\sum_i (\exp(x^*_k-\alpha_k\nabla_xf(x_k,y_k)))_i}, ~\bar{y}_k=\frac{\exp(y_k^*-\alpha_k\nabla_yf(x_k,y_k))}{\sum_i (\exp(y^*_k-\alpha_k\nabla_yf(x_k,y_k)))_i},  \\
&x_{k+1}=\frac{\exp(x_k^*-\alpha_k\nabla_xf(\bar{x}_k,\bar{y}_k))}{\sum_i (\exp(x_k^*-\alpha_k\nabla_xf(\bar{x}_k,\bar{y}_k)))_i}, ~y_{k+1}=\frac{\exp(y_k^*-\alpha_k\nabla_yf(\bar{x}_k,\bar{y}_k))}{\sum_i (\exp(y_k^*-\alpha_k\nabla_yf(\bar{x}_k,\bar{y}_k)))_i},\\
&(x^*_{k+1},y^*_{k+1})=(x^*_k,y^*_k)-\alpha_k(\nabla_xf(\bar{x}_k,\bar{y}_k), -\nabla_yf(\bar{x}_k,\bar{y}_k)).
\end{align}
\end{subequations}
Denote
\begin{subequations}\label{A5}
\begin{align}[left = \empheqlbrace\,]
&\xi_k^x:=\alpha_k\nabla_xf(x_k,y_k)+\alpha_k\beta_k(\nabla_xf(x_k,y_k)-\nabla_xf(x_{k-1},y_{k-1})),\\
&\xi_k^y:=\alpha_k\nabla_yf(x_k,y_k)+\alpha_k\beta_k(\nabla_yf(x_k,y_k)-\nabla_yf(x_{k-1},y_{k-1})).
\end{align}
\end{subequations}
Then, the fourth scheme is obtained by applying the operator extrapolation method \eqref{OE1}:
\begin{equation}\label{A4}
 x_{k+1}=\frac{\exp(\ln x_k-\xi_k^x)}{\sum_i (\exp(\ln x_k-\xi_k^x))_i}, ~y_{k+1}=\frac{\exp(\ln y_k-\xi_k^y)}{\sum_i (\exp(\ln y_k-\xi_k^y))_i}.
\end{equation}
The last one is obtained by applying the EP method \eqref{BEP}:
\begin{subequations}\label{A5}
\begin{align}[left = \empheqlbrace\,]
&x_{k+1}=\frac{\exp(x_k^*-\xi_k^x)}{\sum_i (\exp(x_k^*-\xi_k^x))_i}, ~y_{k+1}=\frac{\exp(y_k^*-\xi_k^y)}{\sum_i (\exp(y_k^*-\xi_k^y))_i},\\
&(x^*_{k+1},y^*_{k+1})=(x^*_k,y^*_k)-(\xi_k^x,\xi_k^y).
\end{align}
\end{subequations}
It should be noted that the choices of parameters $\alpha_k$ and $\beta_k$ depend on the saddle-point functions.

\section{Convergence analysis}
In this section, we deduce a group of convergence results for the proposed mirror frameworks.

\subsection{Convergence analysis for mirror EG method}
This part is devoted to the convergence analysis for the mirror EG method. We first introduce an important lemma.
\begin{lemma}\label{lemMEG}
Suppose that the operator $F$ is $\lambda$-relatively Lipschitz continuous with respect to $\omega$. Let $\{\bar{u}_k, u_k\}_{k\geqslant 0}$ be generated by the mirror extragradient method with the step parameters $\alpha_k$ and $\beta$ satisfying $0<\lambda\alpha_k\leqslant \beta$.
Then, for any $u\in \cU$ we have
\begin{equation}\label{recur1}
\langle \alpha_k F(\bar{u}_k), \bar{u}_k-u \rangle \leqslant D_\omega^{u_{k,2}^*}(u,u_k) - D_\omega^{u_{k+1,2}^*}(u,u_{k+1})+(\beta-1)D_\omega^{u_{k,2}^*}(u_{k+1},u_k).
\end{equation}
Moreover, if $\omega$ is a $\mu_0$-strongly convex function with $L_0$-Lipschitz-continuous gradient, then we have the following tight estimation:
\begin{equation}\label{bound1}
\frac{\mu_0\alpha^2_k}{2L_0^2}\|F(\bar{u}_k)\|_*^2\leqslant D_\omega(u_{k+1},u_k)\leqslant \frac{\alpha^2_k}{2\mu_0}\|F(\bar{u}_k)\|_*^2.
\end{equation}
\end{lemma}

\begin{proof}
It follows from \eqref{MEGa} that
$$\bar{u}_k^*:=u^*_{k,1}-\frac{\alpha_k}{\beta}F(u_k)\in \partial \omega(\bar{u}_k).$$
Noting that $u^*_{k,1}\in \partial \omega(u_k)$ and applying Lemma \ref{lemBreg}, for any $u\in \cU$ we have that
\begin{eqnarray}\label{est1}
\begin{array}{lll}
  \langle \frac{\alpha_k}{\beta} F(u_k), \bar{u}_k-u\rangle &=& \langle u^*_{k,1}-\bar{u}_k^*, \bar{u}_k-u\rangle\\
&= &  D_\omega^{u_{k,1}^*}(u,u_k) - D_\omega^{\bar{u}_k^*}(u,\bar{u}_k) - D_\omega^{u_{k,1}^*}(\bar{u}_k,u_k).
\end{array}
\end{eqnarray}
Using \eqref{MEGd} and applying Lemma \ref{lemBreg} again, for any $u\in \cU$ we have that
\begin{eqnarray}\label{est2}
\begin{array}{lll}
  \langle \alpha_k F(\bar{u}_k), u_{k+1}-u \rangle &\leqslant & \langle u_{k+1,2}^*-u^*_{k,2}, u-u_{k+1}\rangle\\
&= & D_\omega^{u_{k,2}^*}(u,u_k) - D_\omega^{u_{k+1,2}^*}(u,u_{k+1}) - D_\omega^{u_{k,2}^*}(u_{k+1},u_k).
\end{array}
\end{eqnarray}
Combining \eqref{est1} with $u=u_{k+1}$ and \eqref{est2}, we derive that
\begin{eqnarray}\label{est3}
\begin{array}{lll}
&&  \langle \alpha_k F(\bar{u}_k), \bar{u}_k-u \rangle \\
&=& \langle \alpha_k F(\bar{u}_k), u_{k+1}-u \rangle +\langle \alpha_k F(\bar{u}_k), \bar{u}_k-u_{k+1}\rangle\\
&=& \langle \alpha_k F(\bar{u}_k), u_{k+1}-u \rangle +\langle \alpha_k F(u_k), \bar{u}_k-u_{k+1}\rangle +\alpha_k\langle F(\bar{u}_k)-F(u_k), \bar{u}_k-u_{k+1}\rangle\\
&\leqslant &  D_\omega^{u_{k,2}^*}(u,u_k) - D_\omega^{u_{k+1,2}^*}(u,u_{k+1})+\beta D_\omega^{u_{k,1}^*}(u_{k+1},u_k) - D_\omega^{u_{k,2}^*}(u_{k+1},u_k) + \\
& &\alpha_k\langle F(\bar{u}_k)-F(u_k), \bar{u}_k-u_{k+1}\rangle- \beta D_\omega^{u_{k,1}^*}(\bar{u}_k,u_k)-\beta D_\omega^{\bar{u}_k^*}(u_{k+1},\bar{u}_k).
\end{array}
\end{eqnarray}
From \eqref{MEGc} it follows that
\begin{equation}\label{est4}
D_\omega^{u^*_{k,1}}(u_{k+1}, u_k)-D_\omega^{u^*_{k,2}}(u_{k+1}, u_k)=\langle u_{k,2}^*-u_{k,1}^*, u_{k+1}-u_k\rangle \leqslant 0.
\end{equation}
Invoking the relatively Lipschitz continuity of $F$ and noting that $\lambda\alpha_k\leqslant \beta$, we obtain
\begin{equation}\label{est5}
\alpha_k\langle F(\bar{u}_k)-F(u_k), \bar{u}_k-u_{k+1}\rangle\leqslant \beta (D_\omega^{u_{k,1}^*}(\bar{u}_k,u_k)- D_\omega^{\bar{u}_k^*}(u_{k+1},\bar{u}_k)).
\end{equation}
Therefore, \eqref{recur1} follows by combining \eqref{est3}-\eqref{est5}.

Now, let us show \eqref{bound1}. Using \eqref{MEGd} with $u=u_k$ and the strong convexity of $\omega$, we deduce that
\begin{eqnarray}\label{est6}
\begin{array}{lll}
\langle \alpha_k F(\bar{u}_k), u_k-u_{k+1} \rangle &\geqslant& \langle u_{k+1,2}^*-u_{k,2}^*, u_{k+1}-u_k\rangle  \\
&=&  D_\omega^{u^*_{k,2}}(u_{k+1}, u_k)+D_\omega^{u^*_{k+1,2}}(u_k, u_{k+1}) \\
&\geqslant & D_\omega^{u^*_{k,2}}(u_{k+1}, u_k)+ \frac{\mu_0}{2}\|u_k-u_{k+1}\|^2.
\end{array}
\end{eqnarray}
Using the Cauchy-Schwartz inequality, we have
\begin{equation}\label{CS1}
  \langle \alpha_k F(\bar{u}_k), u_k-u_{k+1} \rangle\leqslant  \frac{\mu_0}{2}\|u_k-u_{k+1}\|^2+\frac{\alpha_k^2}{2\mu_0}\|F(\bar{u}_k)\|_*^2.
\end{equation}
Combining \eqref{est6} and \eqref{CS1} and using the implied fact that $\omega$ is differentiable, we obtain
$$D_\omega(u_{k+1},u_k)=D_\omega^{u^*_{k,2}}(u_{k+1}, u_k)\leqslant \frac{\alpha^2_k}{2\mu_0}\|F(\bar{u}_k)\|_*^2.$$
It remains to show the left-hand side inequality of \eqref{bound1}. Note that $u^*_{k,i}=\nabla\omega(u_k), i=1,2$. Using \eqref{MEGb} and the fact that $\nabla\omega^*=(\nabla\omega)^{-1}$, we have
$$\nabla\omega (u_{k+1})=\nabla \omega(u_k)-\alpha_kF(\bar{u}_k).$$
Finally, using the fact that  $\omega$ is a $\mu_0$-strongly convex function with $L_0$-Lipschitz-continuous gradient, we derive that
$$ \|\alpha_kF(\bar{u}_k)\|_*^2=\|\nabla\omega (u_{k+1})-\nabla \omega(u_k)\|_*^2\leqslant L_0^2\|u_{k+1}-u_k\|^2\leqslant \frac{2L_0^2}{\mu_0}D_\omega(u_{k+1},u_k).$$
Note that the inequalities in \eqref{bound1} become equalities when $\omega=\frac{1}{2}\|\cdot\|^2$, i.e., $\mu_0=L_0=1$. Hence, the estimation is tight. This completes the proof.
\end{proof}

Now, we state a sublinear convergence guarantee that generalizes the original result--Theorem 2 in \cite{2007dual} for Nesterov's dual extrapolation method and the newly developed result--Theorem 7.4 in \cite{2019Unifying} for UMP.
\begin{theorem}\label{thMEG}
Suppose that the operator $F$ is monotone and $\lambda$-relatively Lipschitz continuous with respect to $\omega$. Let $\{\bar{u}_k, u_k\}_{k\geqslant 0}$ be the iterates generated by the mirror extragradient method with the step parameters $\alpha_k$ and $\beta$ satisfying $0<\lambda\alpha_k\leqslant \beta\leqslant 1$. Denote $s_t:=\sum_{k=0}^t\alpha_k$ and $\tilde{u}_t:=\frac{1}{s_t}\sum_{k=0}^t\alpha_k\bar{u}_k$.
Then, for any $u\in\cU$ and $t\geqslant 0$ we have
\begin{equation}\label{wbound}
\langle F(u), \tilde{u}_t-u \rangle \leqslant \frac{1}{s_t}D_\omega^{u^*_{0,2}}(u,u_0).
\end{equation}
\end{theorem}

\begin{proof}
Summing up \eqref{recur1} in Lemma \ref{lemMEG} from $k=0$ to $t$ and noting that $\beta\leqslant 1$, we have
\begin{equation}\label{esc1}
  \sum_{k=0}^t\alpha_k\langle F(\bar{u}_{k}), \bar{u}_{k}-u\rangle\leqslant D_\omega^{u^*_{0,2}}(u,u_{0}).
\end{equation}
Using the monotonicity of $F$, we derive that
\begin{eqnarray}\label{est61}
\begin{array}{lll}
 \langle F(u), \tilde{u}_t-u\rangle &= &  \frac{1}{s_t}\sum_{k=0}^t \langle F(u), \alpha_k(\bar{u}_{k}-u)\rangle \\
&\leqslant&\frac{1}{s_t}\sum_{k=0}^t \alpha_k\langle F(\bar{u}_{k}), \bar{u}_{k}-u\rangle\\
&\leqslant&\frac{1}{s_t}D_\omega^{u^*_{0,2}}(u,u_{0}),
\end{array}
\end{eqnarray}
which completes the proof.
\end{proof}

The result above with $\alpha_k\equiv \frac{\beta}{\lambda}$ gives
$$\langle F(u), \tilde{u}_t-u \rangle \leqslant \frac{\lambda D_\omega^{u^*_{0,2}}(u,u_{0})}{\beta(t+1)},$$
which implies that $\tilde{u}_t$ tends to a weak solution $\tilde{u}$ with the convergence rate $\frac{1}{t}$ if $\cU$ is bounded (the boundedness could be dropped; see e.g. \cite{2021Extragradient}).
However, if $F$ is non-monotone, then it will make nonsense to find weak solutions for VIs. The following is a sublinear convergence result on the squared norm of the operator $F$, which generalizes the result--Theorem 3.2 in \cite{2020efficient}.
\begin{theorem}\label{thMEG2}
Let $\omega$ be a $\mu_0$-strongly convex function with $L_0$-Lipschitz-continuous gradient; denote $\kappa_0:=\frac{\mu_0}{L_0}$. Suppose that the operator $F$ is  $\lambda$-relatively Lipschitz continuous with respect to $\omega$ and $\rho$-weakly MVI monotone with $\rho<\frac{\kappa_0}{\lambda}$. Let $\{\bar{u}_k, u_k\}_{k\geqslant 0}$ be the iterates generated by the mirror extragradient method with $\alpha_k\equiv \frac{1}{2\lambda}$ and $\beta=\frac{1}{2}$.
Then, there exists $\hat{u}\in \cU_s$ such that for any $t\geqslant 0$ we have
\begin{equation}\label{wwbound}
\frac{1}{t+1}\sum_{k=0}^t \|F(\bar{u}_k)\|_*^2 \leqslant \frac{16\lambda^2}{(t+1)(\kappa_0-4\lambda\rho)}D_\omega(\hat{u},u_0).
\end{equation}
In particular, we have
\begin{equation}\label{wwbound1}
\min_{0\leqslant k\leqslant t} \|F(\bar{u}_k)\|_*^2 \leqslant \frac{16\lambda^2}{(t+1)(\kappa_0-4\lambda\rho)}D_\omega(\hat{u},u_0).
\end{equation}
\end{theorem}

\begin{proof}
Using \eqref{recur1} in Lemma \ref{lemMEG}, the fact that  $\omega$ is $\mu_0$-strongly convex function with $L_0$-Lipschitz-continuous gradient, and the $\rho$-weakly monotone property, we have
\begin{equation}\label{esd1}
  -\frac{\rho}{4\lambda}\|F(\bar{u}_k)\|_*^2 \leqslant D_\omega(u,u_k) - D_\omega(u,u_{k+1})-\frac{1}{2}D_\omega(u_{k+1},u_k).
\end{equation}
Based on the estimation \eqref{bound1} in Lemma \ref{lemMEG}, we further get
\begin{equation}\label{esd2}
  \frac{\kappa_0-4\lambda\rho}{16\lambda^2}\|F(\bar{u}_k)\|_*^2 \leqslant D_\omega(u,u_k) - D_\omega(u,u_{k+1}).
\end{equation}
Now, summing up \eqref{esd2} from $k=0$ to $t$ gives \eqref{wwbound}. This completes the proof.
\end{proof}

\subsection{Convergence analysis for mirror EP method}
This part is devoted to the convergence analysis for the mirror EP method. We begin with an important lemma as well.
\begin{lemma}\label{lemMEP}
Let $\{u_k\}_{k\geqslant 0}$ be the iterates generated by the mirror extrapolation method. Denote $\Delta F_k: =F(u_k)-F(u_{k-1})$ and
$\Delta D_k(u):= D_\omega^{u_k^*}(u,u_k)-D_\omega^{u_{k+1}^*}(u,u_{k+1}).$ Then, for any $u\in \cU$ and $k\geqslant 0$ we have
\begin{eqnarray}\label{recur2}
\begin{array}{lll}
 \alpha_k\langle  F(u_{k+1}), u_{k+1}-u\rangle&\leqslant & \alpha_k\langle \Delta F_{k+1}, u_{k+1}-u\rangle-\alpha_k\beta_k\langle \Delta F_k, u_k-u\rangle+ \\
 &&\Delta D_k(u)-D_\omega^{u_k^*}(u_{k+1},u_k) +\alpha_k\beta_k\langle \Delta F_k, u_k-u_{k+1}\rangle.
\end{array}
\end{eqnarray}
\end{lemma}

\begin{proof}
On one hand, it follows from \eqref{MEPc} and the notation $\Delta D_k(u)$  that
\begin{eqnarray}\label{est11}
\begin{array}{lll}
  \langle \xi_k, u_{k+1}-u \rangle &\leqslant & \langle u_{k+1}^*-u^*_{k}, u-u_{k+1}\rangle\\
&= & D_\omega^{u_{k}^*}(u,u_k) - D_\omega^{u_{k+1}^*}(u,u_{k+1}) - D_\omega^{u_{k}^*}(u_{k+1},u_k)\\
&=& \Delta D_k(u)- D_\omega^{u_{k}^*}(u_{k+1},u_k).
\end{array}
\end{eqnarray}
On the other hand, using \eqref{MEPa} and the notation $\Delta F_k$, we derive that
\begin{eqnarray}\label{est12}
\begin{array}{lll}
  \langle \xi_k, u_{k+1}-u \rangle &= & \alpha_k\langle  F(u_{k+1}), u_{k+1}-u\rangle-\alpha_k\langle \Delta F_{k+1}, u_{k+1}-u\rangle+  \\
& & \alpha_k\beta_k\langle \Delta F_k, u_k-u\rangle-\alpha_k\beta_k\langle \Delta F_k, u_k-u_{k+1}\rangle.
\end{array}
\end{eqnarray}
Combining \eqref{est11} and \eqref{est12}, we immediately obtain \eqref{recur2}.
\end{proof}

Now, we state the sublinear convergence for the mirror EP method.
\begin{theorem}\label{thMEP}
Suppose that the operator $F$ is monotone and $\lambda$-relatively Lipschitz continuous with respect to $\omega$. Let $\{u_k\}_{k\geqslant 0}$ be generated by the mirror EP method with the initial conditions $u_0=u_{-1}$ and the step parameters $\alpha_k$ and $\beta_k$ satisfying
\begin{eqnarray}\label{para1}
\left\{\begin{array}{lll}
\alpha_k\beta_k=\alpha_{k-1}, \\
\lambda(\alpha_k+\alpha_{k-1})\leqslant 1.
\end{array} \right.
\end{eqnarray}
Denote $s_t:=\sum_{k=0}^t\alpha_k$ and $\tilde{u}_t:=\frac{1}{s_t}\sum_{k=0}^t\alpha_ku_{k+1}$.
Then, for any $u\in\cU$ and $t\geqslant0$ we have
\begin{equation}\label{wbound}
\langle F(u), \tilde{u}_t-u \rangle \leqslant \frac{1}{s_t}D_\omega^{u^*_{0}}(u,u_0).
\end{equation}
\end{theorem}

\begin{proof}
First, using the relatively Lipschitz continuity, we have
\begin{eqnarray}\label{est21}
\begin{array}{lll}
 \langle \Delta F_k,u_k- u_{k+1}\rangle&=&\langle F(u_k)-F(u_{k-1}) ,u_k- u_{k+1}\rangle\\
 &\leqslant &\lambda D_\omega^{u_k^*}(u_{k+1},u_k)+\lambda D_\omega^{u_{k-1}^*}(u_k,u_{k-1}).
\end{array}
\end{eqnarray}
Now, using \eqref{est21}, \eqref{recur2} in Lemma \ref{lemMEP}, and the parameter relations \eqref{para1}, we obtain
\begin{eqnarray}\label{est31}
\begin{array}{lll}
 \alpha_k\langle F(u_{k+1}), u_{k+1}-u\rangle&\leqslant &  \alpha_k\langle \Delta F_{k+1}, u_{k+1}-u\rangle-\alpha_{k-1}\langle \Delta F_k, u_k-u\rangle+\\
 &&D_\omega^{u_k^*}(u,u_k) - D_\omega^{u_{k+1}^*}(u,u_{k+1})+\\
 &&\lambda\alpha_{k-1}D_\omega^{u_{k-1}^*}(u_k,u_{k-1}) - \lambda\alpha_kD_\omega^{u_k^*}(u_{k+1},u_k).
\end{array}
\end{eqnarray}
Summing up \eqref{est31} from $k=0$ to $t$ and noting that $u_0=u_{-1}$, we have
\begin{eqnarray}\label{est41}
\begin{array}{lll}
 \sum_{k=0}^t\alpha_k\langle F(u_{k+1}), u_{k+1}-u\rangle&\leqslant &\alpha_t\langle \Delta F_{t+1}, u_{t+1}-u\rangle+D_\omega^{u_0^*}(u,u_0)- \\
&& D_\omega^{u_{t+1}^*}(u,u_{t+1})- \lambda\alpha_tD_\omega^{u_t^*}(u_{t+1},u_t).
\end{array}
\end{eqnarray}
Using again the relatively Lipschitz continuity and noting that $\lambda \alpha_t\leqslant 1$, we have
$$\alpha_t\langle \Delta F_{t+1}, u_{t+1}-u\rangle \leqslant D_\omega^{u_{t+1}^*}(u,u_{t+1})+ \lambda\alpha_tD_\omega^{u_t^*}(u_{t+1},u_t).$$
Thus, we get
\begin{equation}\label{est51}
  \sum_{k=0}^t\alpha_k\langle F(u_{k+1}), u_{k+1}-u\rangle\leqslant D_\omega^{u_0^*}(u,u_0).
  \end{equation}
Finally, using the monotonicity of $F$, we derive that
\begin{eqnarray}\label{est61}
\begin{array}{lll}
 \langle F(u), \tilde{u}_t-u\rangle &= &  \frac{1}{s_t}\sum_{k=0}^t \langle F(u), \alpha_k(u_{k+1}-u)\rangle \\
&\leqslant&\frac{1}{s_t}\sum_{k=0}^t \alpha_k\langle F(u_{k+1}), u_{k+1}-u\rangle\\
&\leqslant&\frac{1}{s_t}D_\omega^{u^*_{0}}(u,u_0).
\end{array}
\end{eqnarray}
This completes the proof.
\end{proof}

We are now ready to show the linear convergence of the mirror EP method for relatively Lipschitiz and relatively restricted monotone VIs. It should be noted that Theorem \ref{thMEP2} below and its proof are inspired by the linear convergence results in \cite{2020simple}. However, the linear convergence results here hold with a sharper rate under weaker assumptions, and apply to more general algorithmic framework--the mirror EP method.
\begin{theorem}\label{thMEP2}
Suppose that the operator $F$ is $\lambda$-relatively Lipschitz continuous and $\mu$-relatively restricted monotone with respect to $\omega$ and the condition number of $F$ is denoted by $\kappa:=\frac{\mu}{\lambda}$. Let $\{u_k\}_{k\geqslant 0}$ be generated by the mirror EP method with  the initial conditions $u_0=u_{-1}$ and the step parameters $\alpha_k$ and $\beta_k$ satisfying
\begin{subequations}
\begin{align}[left = \empheqlbrace\,]
\label{p1}&\alpha_k\beta_k\gamma_k=\alpha_{k-1}\gamma_{k-1}, \\
\label{p2}&\gamma_k\leqslant \gamma_{k-1}(1+2\mu\alpha_{k-1}),\\
\label{p3}&\lambda(\alpha_k+\alpha_k\beta_k)\leqslant 1,\\
\label{p4}&1+2\alpha_k\mu-\alpha_k\lambda\geqslant 0.
\end{align}
\end{subequations}
Then, for any $u\in\cU$ and $t\geqslant 0$ we have
\begin{equation}\label{rate}
D_\omega^{u^*_{t+1}}(u,u_{t+1}) \leqslant \frac{\gamma_0 D_\omega^{u^*_{0}}(u,u_0)}{\gamma_t(1+2\mu\alpha_t-\lambda\alpha_t)}.
\end{equation}
In particular, if $\alpha_k\equiv\frac{\theta}{2\lambda}$ and $\beta_k\equiv \beta$ with $\theta>0, \beta>0$ that satisfy the conditions \eqref{p1}-\eqref{p4}, then the best convergence rate, obtained by setting $\theta=\theta_0$ and $\beta=\frac{1}{1+\kappa\theta_0}$ with $\theta_0=\frac{\kappa-1+\sqrt{1+\kappa^2}}{\kappa}$, reads as
\begin{equation}\label{bestrate}
D_\omega^{u^*_{t+1}}(u,u_{t+1}) \leqslant (\sqrt{1+\kappa^2}-\kappa)^{t+1}(2+\frac{1}{\kappa})D_\omega^{u^*_{0}}(u,u_0).
\end{equation}
\end{theorem}

\begin{proof}
Multiplying \eqref{recur2} in Lemma \ref{lemMEP} with $\gamma_k$, we obtain
\begin{eqnarray}\label{a1}
\begin{array}{lll}
 \gamma_k\Delta D_k(u) &\geqslant & \gamma_k\alpha_k\langle  F(u_{k+1}), u_{k+1}-u\rangle-\gamma_k\alpha_k\langle \Delta F_{k+1}, u_{k+1}-u\rangle+ \\
 &&\gamma_k\alpha_k\beta_k\langle \Delta F_k, u_k-u\rangle+\gamma_k D_\omega^{u_k^*}(u_{k+1},u_k)-\gamma_k\alpha_k\beta_k\langle \Delta F_k, u_k-u_{k+1}\rangle.
\end{array}
\end{eqnarray}
Summing up \eqref{a1} from $k=0$ to $t$, using the relationship \eqref{p1}, and noting that $\Delta F_0=0$, we derive that
\begin{eqnarray}\label{a2}
\begin{array}{lll}
 \sum_{k=0}^t\gamma_k\Delta D_k(u) &\geqslant & \sum_{k=0}^t\left[\gamma_k\alpha_k\langle  F(u_{k+1}), u_{k+1}-u\rangle\right]-\gamma_t\alpha_t\langle \Delta F_{t+1}, u_{t+1}-u\rangle+S_t,
\end{array}
\end{eqnarray}
where $$S_t:=\sum_{k=0}^t\left[\gamma_k D_\omega^{u_k^*}(u_{k+1},u_k)-\gamma_k\alpha_k\beta_k\langle \Delta F_k, u_k-u_{k+1}\rangle\right].$$
Invoking the relatively Lipschitz continuity to yield
\begin{eqnarray}\label{a3}
\begin{array}{lll}
 \langle \Delta F_k,u_k- u_{k+1}\rangle\leqslant \lambda D_\omega^{u_k^*}(u_{k+1},u_k)+\lambda D_\omega^{u_{k-1}^*}(u_k,u_{k-1}).
\end{array}
\end{eqnarray}
Thus, we can derive that
\begin{eqnarray}\label{a4}
\begin{array}{lll}
 S_t&\geqslant &\sum_{k=0}^t\left[(\gamma_k-\alpha_k\beta_k\gamma_k\lambda)D_\omega^{u_k^*}(u_{k+1},u_k)
 -\alpha_k\beta_k\gamma_k\lambda D_\omega^{u_{k-1}^*}(u_k,u_{k-1})\right]\\
 &=&\sum_{k=0}^t\left[\mu_k D_\omega^{u_k^*}(u_{k+1},u_k)
 -\nu_k D_\omega^{u_{k-1}^*}(u_k,u_{k-1})\right]\\
 &=&\sum_{k=0}^{t-1}\left[(\mu_k-\nu_{k+1}) D_\omega^{u_k^*}(u_{k+1},u_k)\right]+\mu_tD_\omega^{u_t^*}(u_{t+1},u_t),
\end{array}
\end{eqnarray}
where $\mu_k=\gamma_k-\alpha_k\beta_k\gamma_k\lambda$ and $\nu_k=\alpha_k\beta_k\gamma_k\lambda$. Using the relationships \eqref{p1} and \eqref{p3}, we derive that
\begin{eqnarray}\label{a5}
\begin{array}{lll}
 \mu_k-\nu_{k+1}&= &\gamma_k-\alpha_k\beta_k\gamma_k\lambda-\alpha_{k+1}\beta_{k+1}\gamma_{k+1}\lambda\\
 &=&\gamma_k-\alpha_k\beta_k\gamma_k\lambda-\alpha_k\gamma_k\lambda\geqslant0.
\end{array}
\end{eqnarray}
Thereby, it follows from \eqref{a4} that
\begin{equation}\label{a6}
  S_t\geqslant \mu_tD_\omega^{u_t^*}(u_{t+1},u_t).
\end{equation}
 Now, combining \eqref{a2} and \eqref{a6}, we obtain
 \begin{eqnarray}\label{a7}
\begin{array}{lll}
 \sum_{k=0}^t\gamma_k\Delta D_k(u) &\geqslant & \sum_{k=0}^t\left[\gamma_k\alpha_k\langle  F(u_{k+1}), u_{k+1}-u\rangle\right]-\\
 &&\gamma_t\alpha_t\langle \Delta F_{t+1}, u_{t+1}-u\rangle+\mu_tD_\omega^{u_t^*}(u_{t+1},u_t).
\end{array}
\end{eqnarray}
Invoking the relatively Lipschitz continuity again, we derive that
 \begin{eqnarray}\label{a8}
\begin{array}{lll}
 &&-\gamma_t\alpha_t\langle \Delta F_{t+1}, u_{t+1}-u\rangle+\mu_tD_\omega^{u_t^*}(u_{t+1},u_t)\\
 &\geqslant &-\alpha_t\gamma_t\lambda D_\omega^{u_t^*}(u_{t+1},u_t)-\alpha_t\gamma_t\lambda D_\omega^{u_{t+1}^*}(u,u_{t+1})+\mu_tD_\omega^{u_t^*}(u_{t+1},u_t)\\
 &=& (\mu_t-\alpha_t\gamma_t\lambda) D_\omega^{u_t^*}(u_{t+1},u_t)-\alpha_t\gamma_t\lambda D_\omega^{u_{t+1}^*}(u,u_{t+1}).
\end{array}
\end{eqnarray}
Note that from \eqref{p3}, we have
$$ \mu_t-\alpha_t\gamma_t\lambda=\gamma_t-\alpha_t\beta_t\gamma_t\lambda-\alpha_t\gamma_t\lambda\geqslant0.$$
Thus, combining \eqref{a7}, \eqref{a8}, and using the relatively restricted monotonicity, we derive that
\begin{eqnarray}\label{a9}
\begin{array}{lll}
 \sum_{k=0}^t\gamma_k\Delta D_k(u) &\geqslant & \sum_{k=0}^t\left[\gamma_k\alpha_k\langle  F(u_{k+1}), u_{k+1}-u\rangle\right]-\alpha_t\gamma_t\lambda D_\omega^{u_{t+1}^*}(u,u_{t+1})\\
 &\geqslant & \sum_{k=0}^t\left[2\alpha_k\gamma_k\mu D_\omega^{u_{k+1}^*}(u,u_{k+1})\right]-\alpha_t\gamma_t\lambda D_\omega^{u_{t+1}^*}(u,u_{t+1}).
\end{array}
\end{eqnarray}
Now, using the expression of $\Delta D_k(u)$ and \eqref{p2}, we further derive that
\begin{eqnarray}\label{a10}
\begin{array}{lll}
&&\sum_{k=0}^t\left[(\gamma_k+2\alpha_k\gamma_k\mu)D_\omega^{u_{k+1}^*}(u,u_{k+1})\right]-\alpha_t\gamma_t\lambda D_\omega^{u_{t+1}^*}(u,u_{t+1})\\
&\leqslant& \sum_{k=0}^t\gamma_k D_\omega^{u_{k}^*}(u,u_{k})=\gamma_0D_\omega^{u_0^*}(u,u_0)+\sum_{k=1}^t\gamma_k D_\omega^{u_{k}^*}(u,u_{k})\\
&\leqslant& \gamma_0D_\omega^{u_0^*}(u,u_0)+\sum_{k=1}^t(\gamma_{k-1}+2\alpha_{k-1}\gamma_{k-1}\mu) D_\omega^{u_{k}^*}(u,u_{k})\\
&\leqslant& \gamma_0D_\omega^{u_0^*}(u,u_0)+\sum_{k=0}^{t-1}(\gamma_{k}+2\alpha_{k}\gamma_{k}\mu) D_\omega^{u_{k+1}^*}(u,u_{k+1}).
\end{array}
\end{eqnarray}
Thereby,
\begin{equation}\label{a11}
 (\gamma_t+2\alpha_t\gamma_t\mu-\alpha_t\gamma_t\lambda)D_\omega^{u_{t+1}^*}(u,u_{t+1})\leqslant \gamma_0D_\omega^{u_0^*}(u,u_0).
\end{equation}
This together with \eqref{p4} implies \eqref{rate}.

It remains to determine the best convergence rate when the parameters $\alpha_k, \beta_k$ are fixed constants with the form of
$\alpha_k\equiv\frac{\theta}{2\lambda}$ and $\beta_k\equiv \beta$. In such setting, it follows from \eqref{p1} that
$$\gamma_k=\frac{1}{\beta}\gamma_{k-1}=\cdots=\left(\frac{1}{\beta}\right)^k\gamma_0.$$
Now, eliminating $\gamma_k$ from \eqref{p2} and using \eqref{p3}-\eqref{p4}, we get
\begin{subequations}
\begin{align}[left = \empheqlbrace\,]
\label{p5}&\frac{1}{\beta}\leqslant 1+\kappa\theta, \\
\label{p6}&\theta\beta\leqslant 2-\theta,\\
\label{p7}&1+\kappa\theta-\frac{\theta}{2}\geqslant0.
\end{align}
\end{subequations}
From \eqref{p6}, we observe that $\theta\leqslant 2$. Thus, a product of \eqref{p5} and \eqref{p6} shows that the parameter $\theta$ must satisfy the inequality
$$\theta\leqslant (2-\theta)(1+\kappa\theta),$$
from which it must hold that
$$0<\theta\leqslant \theta_0:=\frac{\kappa-1+\sqrt{1+\kappa^2}}{\kappa}.$$
Note that $\theta\leqslant \theta_0<2$ ensures \eqref{p7} to hold. Therefore, $\theta_0$ is the largest possible value of $\theta$ that satisfies \eqref{p5}-\eqref{p7}. Hence, we could use \eqref{p5} and \eqref{p6} to bound $\beta$ as follows.
$$\frac{1}{1+\kappa\theta_0}\leqslant \frac{1}{1+\kappa\theta}\leqslant \beta\leqslant \frac{2}{\theta}-1.$$
Thus, the smallest possible value of $\beta$ is $\frac{1}{1+\kappa\theta_0}$, which corresponds to the best convergence rate by observing that $\frac{\gamma_0}{\gamma_k}=\beta^k$. Finally, after some simple relaxations, we could get
$$\frac{1}{1+2\mu\alpha_t-\lambda\alpha_t}\leqslant2+\frac{1}{\kappa}.$$
Thereby, we obtain \eqref{bestrate}. This completes the proof.
\end{proof}

\begin{remark}
If $\omega$ is differentiable and 1-strongly convex and $F$ is $\lambda$-Lipschitz continuous and $\mu$-relatively restricted monotone with respect to $\omega$, then the authors of \cite{2020simple} obtained the convergence result
\begin{equation*}\label{Lan}
D_\omega(u,u_{t+1}) \leqslant (\frac{1}{1+\kappa})^{t+1}(1+\frac{1}{\kappa})D_\omega^{u^*_{0}}(u,u_0),
\end{equation*}
for the OE method \eqref{OE} by taking $\alpha_t\equiv\frac{1}{2\lambda}, \beta_t\equiv\frac{1}{1+\kappa}$;  the corresponding extrapolation value is $\alpha_t\beta_t=\frac{1}{2\lambda(1+\kappa)}$. Our convergence rate in Theorem \ref{thMEP2} specialized to the OE method \eqref{OE} is better since
$$\sqrt{1+\kappa^2}-\kappa=\frac{1}{\sqrt{1+\kappa^2}+\kappa}<\frac{1}{1+\kappa}.$$
The reason may lie in that we choose larger step sizes $\alpha_t\equiv\frac{\theta_0}{2\lambda}$ and larger extrapolation values since in our setting $$\alpha_t\beta_t=\frac{1}{2\lambda(1+\kappa)}\frac{\theta_0+\kappa\theta_0}{1+\kappa\theta_0}
>\frac{1}{2\lambda(1+\kappa)},$$
by noting $\theta_0> 1$.
\end{remark}

\section{Conclusion}
In this study, we introduced two unified algorithmic frameworks--mirror EG and EP methods for solving relatively Lipschitz and generalized monotone VIs.
The proposed frameworks provide us with equivalent formulations, similar to the original EG method, for the well-known Nemirovski's mirror prox method and Nesterov's dual extrapolation method as well as the Bregman EG method. The equivalent formulations help us clearly see the essential difference between these methods. Theoretically, we are able to analyze (or even improve) the convergence for all these methods in a unified way with currently weakest assumptions. Nevertheless, the convergence theory is far from completion. For example, it is unclear whether one could derive linear convergence for the mirror EG method under the relatively restricted monotonicity, and derive sublinear convergence in the sense of \eqref{wwbound} in Theorem \ref{thMEG2} for the mirror EP method under the weakly MVI monotone assumption. Moreover, the research directions, briefly described below, may be considered as future work as well.

It is well-known that VIs can be reformulated as monotone inclusions. Recently, the author of \cite{2020A} proposed the forward-reflected-backward (FRB) splitting for approximating a solution to monotone inclusions and showed its linear convergence under Lipschitz continuity and strong monotonicity. One of the main contributions made in \cite{2020A} is that the authors could relax the cocoercivity assumption to Lipschitz continuity. Since the FRB splitting shares the same extrapolation technique with the operator extrapolation method in \cite{2020simple}, we wonder whether the Lipschitz continuity and strong monotonicity assumed by FRB for linear convergence can be relaxed into relative versions.

A slightly more general VIs than \eqref{VI} is to find a point $\hat{u}$ such that
\begin{equation}\label{gVI}
  \langle F(\hat{u}), u-\hat{u}\rangle + r(u)-r(\hat{u}) \geqslant 0, \forall u\in\cU,
\end{equation}
where $r$ is a proper lower semicontinuous convex function, usually playing a regularization role. The entropy-regularized constrained saddle-point problem \cite{2014Learning,2021fast} can be reformulated as its special case. We wonder whether our proposed mirror frameworks could be extended to study \eqref{gVI}.

Last but not least, it would be interesting to extend our results to stochastic VIs.

\section*{Acknowledgements}
The authors wish to express their thanks to the anonymous referees and the associate editor for several helpful comments, which allowed us to improve the original presentation. The first author was supported by the National Science Foundation of China (No.11971480), the Natural Science Fund of Hunan for Excellent Youth (No.2020JJ3038), and the Fund for NUDT Young Innovator Awards (No. 20190105).
The second author was supported by the Natural Science Foundation of China (Nos. 11991020, 11631013, 12021001, 11971372 and 11991021) and the Strategic Priority Research Program of Chinese Academy of Sciences (No. XDA27000000).

\section*{Appendix}
\noindent{\bf Proof of Lemma \ref{resclip}.} Fix $u, v, z\in \RR^d$. By Cauchy-Schwartz inequality, Lipschitzness of $F$, and strong convexity of $\omega$, we derive that for any $u^*\in \partial \omega(u), v^*\in \partial \omega(v)$,
\begin{eqnarray*}
\begin{array}{lll}
 & & \langle F(v)-F(u),v-z\rangle \\
 &\leqslant & \|F(v)-F(u)\|_*\|z-v\|\leqslant L\|v-u\|\|z-v\|\\
 &\leqslant & L(\frac{1}{2}\|v-u\|^2 +\frac{1}{2}\|z-v\|^2)
 \leqslant \frac{L}{\mu}(D_\omega^{u^*}(v,u)+ D_\omega^{v^*}(z,v)))
\end{array}
\end{eqnarray*}
from which the result follows.

\bigskip

\noindent{\bf Proof of Lemma \ref{relsmooth}.} Fix $u, v, z\in \RR^d$. By the definition Bregman distance and relative smoothness of $\phi$, we derive that for any $u^*\in \partial \omega(u), v^*\in \partial \omega(v)$,
\begin{eqnarray*}
\begin{array}{lll}
 & & L(D_\omega^{u^*}(v,u)+ D_\omega^{v^*}(z,v))) \\
 &\geqslant & \phi(v)-[\phi(u)+\langle \nabla \phi(u),v-u\rangle]+ \phi(z)-[\phi(v)+\langle \nabla \phi (v), z-v\rangle]\\
 &= & \phi(z)-\phi(u)-\langle \nabla \phi(u),v-u\rangle - \langle \nabla \phi (v), z-v\rangle\\
 &= & D_\phi(z,u)+\langle \nabla \phi(u),z-u\rangle -\langle \nabla \phi(u),v-u\rangle - \langle \nabla \phi (v),z-v\rangle\\
 &=&  D_\phi(z,u)+\langle \nabla\phi(v)-\nabla\phi(u),v-z\rangle.
\end{array}
\end{eqnarray*}
Note that $F=\nabla\phi$ and  $D_\phi(z,u)\geqslant 0$ due to the convexity of $\phi$. The conclusion follows.

\small
\bibliographystyle{plain}

%\bibliography{EEB}

%%%%% CLEAR DOUBLE PAGE!
\newpage{\pagestyle{empty}\cleardoublepage}

\end{document}